\documentclass[11pt]{amsart}
\usepackage[utf8]{inputenc}
\usepackage{amsmath,amsfonts,amsthm,amssymb}
\usepackage{enumitem}
\usepackage{graphicx}
\usepackage{pgf,tikz,pgfplots}
\usetikzlibrary{arrows}
\usetikzlibrary{angles,quotes}
\usetikzlibrary{calc}
\usepgfplotslibrary{fillbetween}
\usetikzlibrary{positioning}
\usetikzlibrary{calc,shapes.geometric}
\usetikzlibrary{decorations.pathreplacing}
\usetikzlibrary{cd}
\usetikzlibrary{positioning}
\usepackage{mathrsfs}
\usepackage{caption,subcaption}
\usepackage{mathtools}
\usepackage{xspace}
\usetikzlibrary{patterns}
\calclayout
\usepackage{soul}
\usepackage{csquotes}

\usepackage{hyperref}
\hypersetup{
  colorlinks   = true,          
  urlcolor     = blue,          
  linkcolor    = green,          
  citecolor   = orange             
}

\newtheorem{theorem}{Theorem}
\newtheorem{corollary}[theorem]{Corollary}

\newtheorem{lemma}[theorem]{Lemma}

\theoremstyle{definition}
\newtheorem{definition}[theorem]{Definition}
\newtheorem{example}[theorem]{Example}

\theoremstyle{remark}
\newtheorem{remark}[theorem]{Remark}

\DeclareMathOperator{\conv}{conv}
\DeclareMathOperator{\pos}{pos}
\DeclareMathOperator{\interior}{int}

\newcommand\cL{{\mathcal L}}

\newcommand\cT{{\mathcal T}}

\newcommand\MM{{\mathbb M}}

\newcommand\RR{{\mathbb R}}
\newcommand\ZZ{{\mathbb Z}}

\newcommand\smallSetOf[2]{\{#1 \colon #2\}}

\newcommand\doi[1]{\href{http://dx.doi.org/#1}{\texttt{doi:#1}}}
\newcommand\polymake{\texttt{polymake}\xspace}

\graphicspath{{Figures/}}

\tikzstyle{lattice} = [draw=red, fill=red]
\tikzstyle{valid} = [draw=red, thick, fill=white]
\tikzstyle{intersect} = [draw=orange, fill=orange]
\tikzstyle{boundary} = [draw=blue, fill=blue]
\tikzstyle{triangle} = [draw=black, thick, fill=blue!20]
\tikzstyle{inequality} = [draw=green, thick]
\tikzstyle{someline} = [draw=black, dashed]

\title{Forbidden Patterns in Tropical Plane Curves}
\author{Michael Joswig \and Ayush Kumar Tewari}

\address[M.~Joswig]{
  Technische Universität Berlin,
  Chair of Discrete Mathematics/Geometry;
  Max-Planck Institute for Mathematics in the Sciences, Leipzig
}
\email{joswig@math.tu-berlin.de}

\address[A.~K.~Tewari]{
  Technische Universität Berlin,
  Chair of Discrete Mathematics/Geometry
}
\email{tewari@math.tu-berlin.de}

\makeatletter
\@namedef{subjclassname@1991}{Mathematics Subject Classification (2020)}
\makeatother



\thanks{%
  M.~Joswig has been supported by Deutsche Forschungsgemeinschaft (EXC 2046: \enquote{MATH$^+$}, SFB-TRR 195: \enquote{Symbolic Tools in Mathematics and their Application}, and GRK 2434: \enquote{Facets of Complexity}).
  A.~K.~Tewari has been supported by Deutsche Forschungsgemeinschaft (SFB-TRR 195: \enquote{Symbolic Tools in Mathematics and their Application})}

\begin{document}

\begin{abstract}
  Tropical curves in $\RR^2$ correspond to metric planar graphs but not all planar graphs arise in this way.
  We describe several new classes of graphs which cannot occur.
  For instance, this yields a full combinatorial characterization of the tropically planar graphs of genus at most five.
\end{abstract}

\maketitle

\section{Introduction}

A tropical plane curve is a metric graph, $G$, embedded in $\RR^2$, which is dual to a regular subdivision $\Delta$ of some lattice polygon $P$.
In the smooth case $\Delta$ is a unimodular triangulation of $P$, and the planar graph $G$ is trivalent of genus $g$, where $g$ agrees with the number of interior lattice points of $P$.
Such graphs have been called \emph{tropically planar} by Coles et al.~\cite{M19}.
Here we are concerned with the question: Which graphs $G$ occur in this way?

For any fixed $g\geq 1$, Castryck and Voight gave a procedure for finding finitely many lattice polygons with exactly $g$ interior lattice points whose triangulations suffice to produce all such graphs \cite{CV09}.
This was employed in \cite{JB15} and \cite{M19} to computationally determine the tropically planar graphs of genus $g\leq 7$.
Despite its success that approach is severely limited by the combinatorial explosion of the number of planar graphs and the number of triangulations of the relevant polygons.
While it is doubtful that this can be pushed much further, here we seek to find obstructions for a given trivalent graph to be realizable as a tropical plane curve.
It was known previously that graphs with a \emph{sprawling node} \cite[Proposition~4.1]{CartwrightManjunathYao:2016}, \emph{crowded} graphs \cite[Lemma 3.5]{M17} and \emph{TIE-fighter} graphs \cite[Theorem 3.4]{M19} cannot occur; cf.\ Figure~\ref{fig:known-obstructions}.
Here we add further forbidden patterns to this list.
As an immediate consequence, for the first time, we provide a complete combinatorial characterization of the tropically planar graphs of genus $g\leq 5$, in contrast to the proof in \cite{JB15} which rests on substantial computer support.
More importantly, we identify a particularly interesting \enquote{boundary case}:
if it is realizable, a \emph{graph with a heavy cycle} forces a number of geometric restrictions for the polygon $P$ and the triangulation~$\Delta$.
As an example, this leads us to studying a natural family of triangulations, which we call \emph{anti-honeycombs}.
In this way we get a glimpse of tropically planar graphs of genus $g=8$ and beyond.

While this work is motivated by the desire to understand tropical curves and their moduli, here we concentrate on the combinatorial aspects,
which means we do not take edge lengths of metric graphs into account.
This comes with the advantage that our results also hold for triangulations which are not regular.

\paragraph{\bf Acknowledgment}
We are grateful to Dominic Bunnett and Marta Panizzut for helpful comments on an early draft of this article and for requiring a minor correction.

\section{Lattice Polygons and Tropical Plane Curves}

Let $V\subset\RR^d$ be a finite set of points.
A \emph{polytopal subdivision} of $V$ is a polytopal complex which covers the convex hull $\conv V$ and uses (a subset of) the given point set $V$ as its vertices.
If that subdivision is induced by a height function on $V$, it is called \emph{regular}.
A subdivision which consists of simplices is a \emph{triangulation}.
A comprehensive reference on the subject is the monograph~\cite{Triangulations} by De Loera, Rambau and Santos.
Here we will be concerned with the following case.
Let $P$ be a (convex) \emph{lattice polygon}, i.e., $P$ is the convex hull of finitely many points in $\ZZ^2$, and $V=P\cap\ZZ^2$ is the set of lattice points in $P$.
A triangulation $\Delta$ of $V$ is \emph{unimodular} if each triangle in $\Delta$ has normalized area one, i.e., Euclidean area $\tfrac{1}{2}$.
This is the case if and only if $\Delta$ uses all points in $V$.
We write $\partial P$ for the boundary of $P$ and $\interior P$ for its interior.
The number $g(P)=\#(\interior P)$ of interior lattice points is the \emph{genus} of $P$; cf.\ Figure~\ref{fig:antihoney} for an example.

Let $\Delta$ be a not necessarily unimodular triangulation of $V$.
The \emph{dual graph} $\Gamma=\Gamma(\Delta)$ is the abstract graph whose nodes are the triangles of $\Delta$; they form an edge in $\Gamma$ if two triangles share an edge in $\Delta$.
The dual graph is necessarily connected and planar, and each node has degree at most three.

\begin{figure}[th]
  \begin{tikzpicture}[scale=0.5]
  
  \draw[] (30:0) -- (30:3);
  \draw[] (150:0) -- (150:3);
  \draw[] (270:0) -- (270:2.5);
  
  \draw[fill = gray!40!white, draw=black] (1.5,1.2) rectangle ++(1.5,1.5);
  \draw[fill = gray!40!white, draw=black] (-3,1.2) rectangle ++(1.5,1.5);
  \draw[fill = gray!40!white, draw=black] (-0.75,-3.5) rectangle ++(1.5,1.5);
  
  \fill[black] (30:0) circle (.1cm) node[align=left,   above]{};

  \end{tikzpicture}
  \qquad
  \begin{tikzpicture}[scale=0.5]
  
  \draw[] (-1.5,1.5) -- (1,4);
  \draw[] (1,1.5) -- (1,4);
  \draw[] (3.5,1.5) -- (1,4);
  \draw[] (-1.5,0) -- (1,-2.5);
  \draw[] (1,0) -- (1,-2.5);
  \draw[] (3.5,0) -- (1,-2.5);
  
  \draw[fill = gray!40!white, draw=black] (-2,0) rectangle ++(1,1.5);
  \draw[fill = gray!40!white, draw=black] (0.5,0) rectangle ++(1,1.5);
  \draw[fill = gray!40!white, draw=black] (3,0) rectangle ++(1,1.5);
  
  \fill[black] (1,4) circle (.1cm);
  \fill[black] (1,-2.5) circle (.1cm);
  
  \end{tikzpicture}
  \qquad
  \begin{tikzpicture}[scale=0.5]
  \draw[] (0,0) -- (-1,0);
  \draw[] (-1,-1) -- (-2,-1);
  \draw[] (-1,1) -- (-1,0);
  \draw[] (-1,0) -- (-1,-1);
  \draw[] (0,0) -- (1,1);
  \draw[] (0,0) -- (1,-1);
  \draw[] (1,1) -- (1,2);
  \draw[] (3,1) -- (4,0);
  \draw[] (1,-1) -- (3,-1);
  \draw[] (1,-1) -- (1,-2);
  \draw[] (1,-2) -- (3,-2);
  \draw[] (3,-1) -- (3,-2);
  \draw[] (3,-1) -- (4,0);
  \draw[] (4,0) -- (5,0);
  \draw[] (5,0) -- (5,1);
  \draw[] (5,0) -- (5,-1);
  \draw[] (5,1) -- (6,1);
  \draw[] (5,-1) -- (6,-1);
  \draw[] (6,1) -- (6,-1);
  \fill[gray!40!white] (1,1) rectangle (3,2);
  \draw[] (1,1) -- (1,2);
  \draw[] (1,1) -- (3,1);
  \draw[] (3,1) -- (3,2);
  \draw[] (1,2) -- (3,2);
  \fill[gray!40!white] (1,-1) rectangle (3,-2);
  \draw[] (1,-1) -- (1,-2);
  \draw[] (1,-1) -- (3,-1);
  \draw[] (1,-2) -- (3,-2);
  \draw[] (3,-1) -- (3,-2);
  \fill[gray!40!white] (5,-1) rectangle (6,1);
  \draw[] (5,-1) -- (5,1);
  \draw[] (5,1) -- (6,1);
  \draw[] (6,1) -- (6,-1);
  \draw[] (5,-1) -- (6,-1);
  \fill[black] (0,0) circle (.1cm);
  \fill[black] (4,0) circle (.1cm);
  \fill[gray!40!white] (-2,-1) rectangle (-1,1);
  \draw[] (-2,-1) -- (-2,1);
  \draw[] (-2,1) -- (-1,1);
  \draw[] (-1,1) -- (-1,-1);
  \draw[] (-1,-1) -- (-2,-1);
  \end{tikzpicture}
  \caption{Graph with a sprawling node (left), a crowded graph (center), and a TIE-fighter graph (right). Each box represents some subgraph of positive genus.}
  \label{fig:known-obstructions}
\end{figure}
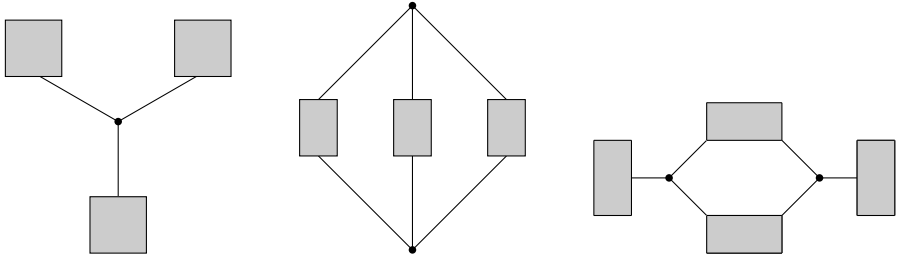

Next we describe a procedure to obtain a certain graph minor of $\Gamma$.
First, if there is a node of degree one, we delete it together with the unique incident edge.
We repeat this step until no nodes of degree one are left.
The remaining edges are \emph{nonredundant}.
Second, if there is a node of degree two, we delete the node and join its neighbors by an edge.
Again we repeat until there are no more nodes of degree two.
The resulting graph is the \emph{skeleton}, which we denote $G=G(\Delta)$.
By construction the skeleton is a trivalent planar graph, and it does not depend on the ordering in which the edge deletions and contractions are performed.
In this way each edge of $G$ arises as an edge path in~$\Gamma$.
This yields a surjective map, which we denote $\eta$, from the nonredundant edges of $\Gamma$ onto the edges of $G$.
Note that this \emph{contraction map} $\eta$ is undefined for any edge which is redundant.
This elementary description of the skeleton is algorithmic in nature and serves our purposes.

A \emph{split} is a subdivision of $V$ with exactly two maximal cells; it is necessarily regular \cite[Lemma~3.5]{HJ08}.
If $U$ and $W$ are the two maximal cells of a split, then the intersection $U\cap W$ is a common edge of the two convex polygons $U$ and $W$.
It spans the corresponding \emph{split line}.
Two splits of $V$ are \emph{weakly compatible} if there is a triangulation which refines them both.
Moreover, two splits are \emph{(strongly) compatible} if their split lines do not meet in the interior of~$P$.
Compatibility implies weak compatibility.
The split lines of two weakly compatible splits which are not strongly compatible must meet in a point in $V$, i.e., an interior lattice point of $P$.
An edge $e$ of the connected graph $G$ is a \emph{cut edge} if deleting $e$ creates two connected components.
Otherwise $e$ lies in some cycle of $G$.

The following three technical results are extracted from the proof of \cite[Theorem 3.4]{M19}, where cut edges are called \enquote{bridges}.
\begin{lemma}\label{lem:cut-split}
  Let $e$ be a cut edge of $G$.
  Then $\eta^{-1}(e)$ contains a cut edge of\/~$\Gamma$, which is dual to an edge, $s$, of $\Delta$.
  Moreover, the vertices of $s$ lie on the boundary $\partial P$, and $s$ spans a split line of $V$.
\end{lemma}

Note that $\eta^{-1}(e)$ in Lemma~\ref{lem:cut-split} may contain several cut edges of $\Gamma$, and so the edge $s$ of the triangulation $\Delta$ may not be unique.

\begin{lemma}\label{lem:cycle-vertex}
  Let $e$ be an edge between $v_{1}$ and $v_{2}$ in a cycle, $C$, of $G$.
  Further, let $T_1$ and $T_2$ be triangles in $\Delta$ which correspond to $v_{1}$ and $v_{2}$ on the path $\eta^{-1}(e)$.
  Then $T_1$ and $T_2$ share an interior lattice point of $P$, and this is dual to $C$.
\end{lemma}

\begin{lemma}\label{lem:area}
  Let $v\in V$ be some lattice point in $P$ with two incident triangles $T_1=\conv\{z,a_1,b_1\}$ and $T_2=\conv\{z,a_2,b_2\}$ both of which are in $\Delta$.
  Further, let $L_i$ be the line spanned by $a_i$ and $b_i$, for $i=1,2$.
  Suppose that $L_1$ and $L_2$ meet in some point, say $w$, such that $a_i$ is closer to $w$ than $b_i$, for $i=1,2$.
  Then the interior of the quadrangle $\conv\{z,a_1,w,a_2\}$ does not contain a point in~$V$, unless $a_1=a_2=w$.
\end{lemma}

With this we can make a small first step to our main results.

\begin{lemma}\label{lem:compatible}
  Splits corresponding to distinct cut edges are compatible.
\end{lemma}
\begin{proof}
  From Lemma~\ref{lem:cut-split} the cut edges $e_1$ and $e_2$ yields two split lines, $S_1$ and~$S_2$, which may not be unique.
  Unless $S_1$ and $S_2$ are compatible, they must meet in a point of the point configuration $V$, which does not lie on $\partial P$, i.e., an interior lattice point of $P$.
  Yet, by Lemma~\ref{lem:cut-split} the line $S_1$ (and similarly $S_2$) contains precisely two points of $V$, and neither lies in the interior.
\end{proof}

The connection to tropical geometry comes about as follows.
See the book by Maclagan and Sturmfels \cite{Tropical+Book} for the general context and \cite{JB15} for a detailed analysis of the case which is of interest here.
Let $\Phi$ be a bivariate tropical polynomial.
Its vanishing locus, the \emph{tropical hypersurface} $\cT=\cT(\Phi)$, is a tropical plane curve in~$\RR^2$.
The Newton polygon of $\Phi$ is a lattice polygon, and the terms of $\Phi$ correspond to lattice points in~$P$.
Moreover, the coefficients of $\Phi$ induce a regular subdivision, $\Delta$, on those lattice points which are dual to $\cT$, i.e., $\cT$ essentially agrees with $\Gamma(\Delta)$.
The tropical plane curve $\cT$ is \emph{smooth} if its dual subdivision $\Delta$ is unimodular.
Conversely, each unimodular regular triangulations of $P$ arises in this way.
In that case the edges of $G$ (with their lengths, which we do not need here) describe $\cT$ as a point in the moduli space $\MM_g^{\text{planar}}$ of tropical plane curves of genus~$g$.
Skeleta of unimodular regular triangulations of lattice polygons are \emph{tropically planar} or \enquote{troplanar} graphs in \cite{M19}.

\begin{example}\label{exmp:anti-honey-2}
  We consider the \emph{anti-honeycomb} triangle
  \begin{equation}\label{eq:anti-honey-2}
    A_{(-2,0;-2,0;-2,0)} \ = \ \conv\{(2,2),(-2,0),(0,-2)\} \enspace,
  \end{equation}
  which occurs as $Q_3^{(4)}$ in \cite{JB15}.
  The genus is $g(A_{(-2,0;-2,0;-2,0)})=4$.
  We call the unimodular triangulation $\Delta_{(-2,0;-2,0;-2,0)}$, shown in Figure~\ref{fig:antihoney} (left), the \emph{anti-honeycomb triangulation} of $A_{(-2,0;-2,0;-2,0)}$.
  Its skeleton, shown in Figure~\ref{fig:antihoney} (right) and called $(303)$ in \cite{JB15}, features three cut edges which correspond to three compatible splits of $A_{(-2,0;-2,0;-2,0)}$.
  The triangulation $\Delta_{(-2,0;-2,0;-2,0)}$ is regular, whence it defines a moduli cone of tropical plane curves.
  That cone is 7-dimensional, while the entire moduli space $\MM_4^{\text{planar}}$ has dimension nine; see \cite[Table~4]{JB15}.
  See Section~\ref{sec:anti-honey} for a more comprehensive discussion of anti-honeycomb polygons, their triangulations and the notation \eqref{eq:anti-honey-2}.
\end{example}

\begin{figure}[th]\centering
  \begin{tikzpicture}[scale=0.78]
    \draw[] (-2,0) -- (2,2);
    \draw[] (-2,0) -- (0,-2);
    \draw[] (0,-2) -- (2,2);
    \draw[] (2,2) -- (1,1);
    \draw[] (1,1) -- (1,0);
    \draw[] (1,1) -- (0,1);
    \draw[] (0,1) -- (1,0);
    \draw[] (0,1) -- (0,0);
    \draw[] (1,0) -- (0,0);
    \draw[] (0,1) -- (-1,0);
    \draw[] (0,1) -- (-1,-1);
    \draw[] (1,0) -- (-1,-1);
    \draw[] (1,0) -- (0,-1);
    \draw[] (-1,0) -- (-2,0);
    \draw[] (-1,0) -- (-1,-1);
    \draw[] (1,0) -- (0,-1);
    \draw[] (0,-1) -- (-1,-1);
    \draw[] (0,-1) -- (0,-2);
    \draw[] (0,0) -- (-1,-1);
    \fill[black] (0,-2) circle (.07cm);
    \fill[black] (-2,0) circle (.07cm);
    \fill[black] (2,2) circle (.07cm);
    \fill[black] (0,1) circle (.07cm);
    \fill[black] (1,0) circle (.07cm);
    \fill[black] (0,0) circle (.07cm);
    \fill[black] (1,1) circle (.07cm);
    \fill[black] (-1,-1) circle (.07cm);
    \fill[black] (0,-1) circle (.07cm);
    \fill[black] (-1,0) circle (.07cm);
  \end{tikzpicture}
  \qquad
  \begin{tikzpicture}[scale=0.22]
    \draw[blue] (0,0) -- (2,0);
    \draw[blue] (0,0) -- (0,2);
    \draw[blue] (0,2) -- (2,0);
    \draw[] (0,0) -- (-2,-2);
    \draw[] (0,2) -- (-1.5,5);
    \draw[orange] (2,0) -- (5,-1.5);
    \draw[red] (5,-1.5) -- (7,-1.5);
    \draw[red] (5,-1.5) -- (7,-3.5);
    \draw[red] (7,-1.5) -- (7,-3.5);
    \draw[brown] (7,-3.5) -- (8.5,-6.5);
    \draw[green] (8.5,-6.5) -- (8.5,-8.5);
    \draw[green] (8.5,-8.5) -- (10.5,-8.5);
    \draw[green] (8.5,-6.5) -- (10.5,-8.5);
    \draw[] (8.5,-8.5) -- (6.5,-10.5);
    \draw[] (10.5,-8.5) -- (13.5,-10);
    \draw[pink] (7,-1.5) -- (10,1.5);
    \draw[magenta] (10,1.5) -- (10,3.5);
    \draw[magenta] (10,1.5) -- (12,1.5);
    \draw[magenta] (10,3.5) -- (12,1.5);
    \draw[] (10,3.5) -- (8.5,6.5);
    \draw[] (12,1.5) -- (15,0);
  \end{tikzpicture}
  \qquad
  \begin{tikzpicture}[scale=0.52]
    \draw[red] (30:1) -- (150:1) -- (270:1) -- cycle;
    
    \draw[pink] (30:1) -- (30:3);
    \draw[orange] (150:1) -- (150:3);
    \draw[brown] (270:1) -- (270:3);
	
    \filldraw[fill=white,draw=magenta] (30:3) circle (0.5cm);
    \filldraw[fill=white,draw=blue] (150:3) circle (0.5cm);
    \filldraw[fill=white,draw=green] (270:3) circle (0.5cm);
  \end{tikzpicture}
  \caption{Anti-honeycomb triangulation $\Delta_{(-2,0;-2,0;-2,0)}$ of genus~4 (left), its dual graph (center), and the corresponding skeleton (right)}
  \label{fig:antihoney}
\end{figure}
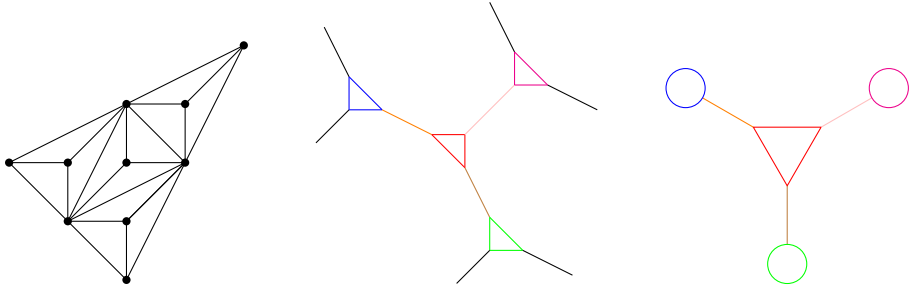

We now sketch the forbidden patterns which are known already.
A node in $G$ is \emph{sprawling} if its removal leaves three connected components; cf.\ Figure~\ref{fig:known-obstructions} (left).
This obstruction to tropical planarity was identified in \cite[Proposition~4.1]{CartwrightManjunathYao:2016} and \cite[Proposition 8.3]{JB15}.
Note that graphs with a sprawling node were called \enquote{sprawling} in \cite{JB15} and \cite{M19}.
A planar embedding of a graph $G$ is called \emph{crowded} if either: there exist two bounded regions sharing at least two edges; or, there exists a bounded region sharing an edge with itself.
If all planar embeddings of $G$ are crowded, then $G$ itself is said to be \emph{crowded}.
In \cite[Lemma 3.5]{M17}, it is shown that crowded graphs can never be tropically planar.
Additionally, \cite[Corollary 3.7]{M17} describes a family of crowded graphs, and this is depicted in Figure~\ref{fig:known-obstructions} (center).
A graph is a \emph{TIE-fighter} if it looks like the one in Figure~\ref{fig:known-obstructions} (right).
TIE-fighter graphs can never be tropically planar was shown in \cite[Theorem 3.4]{M19}.

\section{Heavy Cycles and Sprawling Triangles}

Let $P$ be a lattice polygon with precisely $g$ interior lattice points.
Moreover, let $\Delta$ be a unimodular triangulation of $P$ with dual graph $\Gamma=\Gamma(\Delta)$ and skeleton $G(\Delta)=G$.
The contraction map $\eta$ sends nonredundant edges of $\Gamma$ to edges of $G$.
The $g$ interior lattice points of $P$ bijectively correspond to the bounded regions of the planar graph $G$.
By Euler's formula we have $g = m - n + 1$, where $m$ and $n$ are the numbers of edges and nodes of $G$, respectively.
Our arguments in this section do not require $\Delta$ to be regular.
That is, our results extend to a class of planar graphs, which is slightly more general than tropical plane curves.

\begin{lemma}\label{lem:abz}
  Assume that $P$ contains a unimodular triangle with vertices $a,b,z$ such that neither $a$ nor $b$ is a vertex of $P$, and $z$ is an interior lattice point.
  If $a$ and $b$ lie on $\partial P$ then either $a$ and $b$ lie on a common edge of $P$ or the lattice point $a+b-z$ is contained in $P$.
\end{lemma}
\begin{proof}
  Up to unimodular equivalence we may assume $a=(1,0)$, $b=(0,1)$ and $z=(0,0)$.
  Let $L$ and $M$ be the lines spanned by the edges containing $a$ and $b$, respectively.
  Suppose that $(1,1)$ is not contained in $P$.
  Since $P$ is a lattice polygon, then the intersection of $P$ with the positive orthant agrees with the unit triangle $abz$.
  This entails that $L$ and $M$ agree, which means that $L=M$ defines an edge of $P$, and this contains $a$ as well as $b$.
\end{proof}

\begin{figure}[th]\centering
  \begin{tikzpicture}[scale=0.9]
    \fill[gray!40!white] (0.5,-0.5) rectangle (2.5,-1.5);
    \fill[gray!40!white] (-2.5,2.5) rectangle (-0.5,1.5);
    \fill[gray!40!white] (3.5,2.5) rectangle (5.5,1.5);
    \draw[] (0,0) -- (-0.5,0.5);
    \draw[] (0,0) -- (0.5,-0.5);
    \draw[] (-0.5,0.5) -- (3.5,0.5);
    \draw[] (-1.5,1.5) -- (-0.5,0.5);
    \draw[] (3.5,0.5) -- (4.5,1.5);
    \draw[] (0.5,-0.5) -- (2.5,-0.5);
    \draw[] (0.5,-0.5) -- (0.5,-1.5);
    \draw[] (2.5,-0.5) -- (2.5,-1.5);
    \draw[] (2.5,-1.5) -- (0.5,-1.5);
    \draw[] (2.5,-0.5) -- (3,0);
    \draw[] (3,0) -- (3.5,0.5);
    \draw[] (5.5,1.5) -- (3.5,1.5);
    \draw[] (5.5,2.5) -- (3.5,2.5);
    \draw[] (5.5,1.5) -- (5.5,2.5);
    \draw[] (3.5,1.5) -- (3.5,2.5);
    \draw[] (-2.5,1.5) -- (-0.5,1.5);
    \draw[] (-2.5,2.5) -- (-0.5,2.5);
    \draw[] (-2.5,1.5) -- (-2.5,2.5);
    \draw[] (-0.5,1.5) -- (-0.5,2.5);
    \fill[black] (-0.5,0.5) circle (.09cm) node[align=right, above]{};
    \fill[black] (-0.3,0.5) circle (.000001cm) node[align=right, above]{$v_{2}$};
    \fill[black] (3.5,0.5) circle (.09cm) node[align=right,   above]{$v_{1}\quad$};
    \fill[black] (1.6,-0.25) circle (.0001cm) node[align=left,   above]{$C\quad$};
    \fill[black] (-1.1,0.7) circle (.0001cm) node[align=left,   above]{$e_{2}\quad$};
    \fill[black] (4.5,0.6) circle (.0001cm) node[align=left,   above]{$e_{1}\quad$};
    \fill[black] (1.65,-1.25) circle (.0001cm) node[align=left,   above]{$G_{3}\quad$};
    \fill[black] (-1.3,1.7) circle (.0001cm) node[align=left,   above]{$G_{2}\quad$};
    \fill[black] (4.7,1.7) circle (.0001cm) node[align=left,   above]{$G_{1}\quad$};
  \end{tikzpicture}
  \caption{Graph with the heavy cycle $C$}
  \label{fig:heavy}
\end{figure}
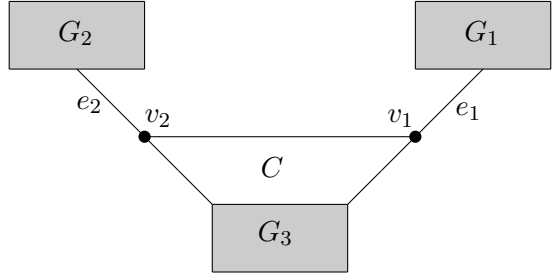

\begin{definition}\label{def:heavy}
  We say that the planar graph $G$ has a \emph{heavy cycle} if it looks like the graph in Figure~\ref{fig:heavy}.
  That is, $G$ has a cycle, $C$, with two vertices, $v_1$ and $v_2$, such that $v_i$ is incident with a cut edge $e_i$ connecting $v_i$ with subgraphs $G_i$ of positive genus;
  a third subgraph, $G_3$, also of positive genus, which shares at least one vertex with the cycle $C$.
\end{definition}
In particular, a graph with a heavy cycle has genus at least four.
From the classification of hyperelliptic graphs in \cite{M17}, we infer that a graph with a heavy cycle is not hyperelliptic.
Moreover, it follows from Lemma~\ref{lem:cut-split} that there are split lines, $S_1$ and $S_2$, dual to the edges $e_1$ and $e_2$ of $G$.
While the split lines may not be unique we just pick some.
By Lemma~\ref{lem:compatible} the splits $S_1$ and $S_2$ are compatible, and thus $P$ decomposes into a union of three lattice polygons $P_1$, $P_2$ and $P'$ such that $\Delta$ induces triangulations of all three.
In this way, we get triangulations $\Delta_1$, $\Delta_2$ and $\Delta'$ such that the cycle $G_i$ is the skeleton of $\Delta_i$ for $i=1,2$, and $G_3\cup C$ is the skeleton of $\Delta'$.
The triangles in $\Delta'$ which are dual to $v_1$ and $v_2$ are denoted $T_1$ and $T_2$, respectively.
We refer to the polygon $P'$ as the \emph{heavy component} of $P$, and likewise $\Delta'$ is the \emph{heavy component} of $\Delta$.
Expressed in the language of \cite{M19}, the heavy component $\Delta'$ arises from $\Delta$ via \enquote{bridge-reduction}.

The following lemma is the technical core of this paper.
Its proof is a bit cumbersome, with several cases to distinguish.
However, it is rather powerful as it delineates a fine border between trivalent planar graphs which are realizable as skeleta of tropically planar curves and those which are not.

\begin{lemma}\label{lem:heavy}
  Suppose that $G$ has a heavy cycle with cut edges $e_1$ and $e_2$ as in Figure~\ref{fig:heavy}.
  Then the triangles $T_1$ and $T_2$ in $\Delta$ share an edge $[z,w]$, where $z$ is the interior lattice point dual to $C$, and the split lines $S_1$ and $S_2$ intersect in $w$, which is a vertex of $P'$, and which lies in the boundary of $P$.
\end{lemma}
\begin{proof}
  By Lemma~\ref{lem:cycle-vertex} the triangles $T_1$ and $T_2$ share a vertex $z$, which is the interior lattice point in $P$ dual to the heavy cycle $C$.
  Up to a unimodular transformation we may assume that $z = (0,0)$, and $T_1=\conv\{z,(1,0),(0,1)\}$.

  Due to Lemma~\ref{lem:abz} the point $(1,1)=(1,0)+(0,1)-z$ lies in $P$.
  We want to show that $(1,1)$ is an interior lattice point of $P_1$ (and $P$).
  First, since $(1,0)$ is in the boundary and $z$ is in the interior of $P$, there are no interior lattice points on the ray $(1,0)+\pos\{(1,0)\}$.
  Similarly, there are no interior lattice points on $(0,1)+\pos\{(0,1)\}$.
  However, since $G_1$ has positive genus at least one interior lattice point of $P_1$ exists, and it must lie in the translated orthant $(1,1)+\pos\{(1,0),(0,1)\}$.
  As $(1,0)$ and $(0,1)$ lie in $P_1$ it follows from the convexity of $P_1$ that $(1,1)$ must be an interior lattice point.
  Now consider the triangle $T_2=\conv\{z,(\alpha,\beta),(\gamma,\delta)\}$ where $\alpha,\beta,\gamma,\delta\in\ZZ$ with
  \begin{equation}\label{eq:det}
    \alpha\delta-\beta\gamma \ = \ 1 \enspace .
  \end{equation}

  Suppose $T_1$ and $T_2$ do not share an edge.
  We are aiming for a contradiction, which then establishes a proof of this lemma.
  Since $(1,1)\in\interior P_1$ the horizontal line $(0,1)+\RR(1,0)$ intersects $P'$ only in $(0,1)$.
  Similarly, the vertical line  $(1,0)+\RR(0,1)$ intersects $P'$ only in $(1,0)$.
  As $(\alpha,\beta)$, $(\gamma,\delta)$, $(1,0)$, and $(0,1)$ are pairwise distinct this yields $\alpha,\beta,\gamma,\delta\leq0$.
  
  We distinguish two cases, depending on whether the two split lines are parallel or not.
  First, suppose the split line $S_2$ is parallel to $S_1$.
  As $T_2$ is a unimodular triangle $S_2$ is the line through $(-1,0)$ and $(0,-1)$.
  Then, $S_{2}$ and $\alpha,\beta,\gamma,\delta \leq 0$ forces $\alpha=\delta=-1$ and $\beta=\gamma=0$.
  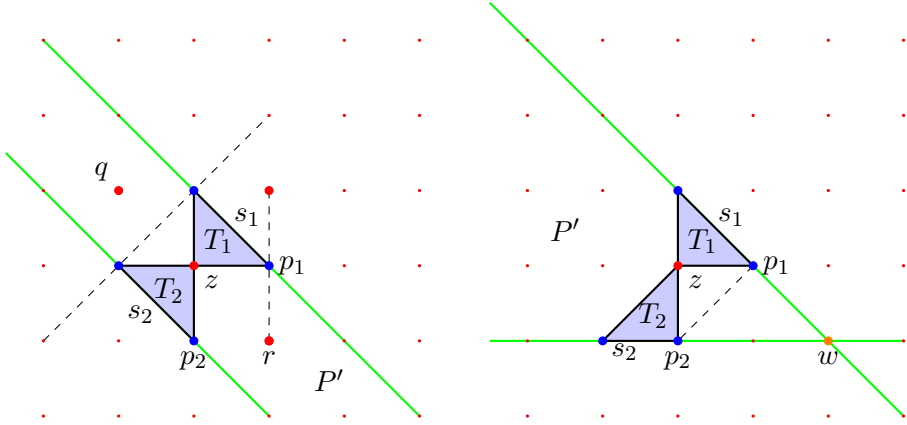
\begin{figure}\centering
    \begin{tikzpicture} 	
      \coordinate (z) at (0,0);
      \coordinate (p1) at (1,0);
      \coordinate (q1) at (0,1);
      \coordinate (p2) at (0,-1);
      \coordinate (q2) at (-1,0);
      \coordinate (u1) at (1,1);
      \coordinate (v1) at (1,-1);
      \coordinate (r) at (1,-1);
      
      \draw[inequality] (-2,3) -- (3,-2); 
      \draw[inequality] (-2.5,1.5) -- (1,-2); 
      \draw[someline] (1,1) -- (1,-1);
      \draw[someline] (-2,-1) -- (1,2);

      \filldraw[triangle] (z) -- (p1) -- (q1) -- cycle;
      \filldraw[triangle] (z) -- (p2) -- (q2) -- cycle;

      \node at ($.33*(z)+.33*(p1)+.33*(q1)$) {$T_1$};
      \node at ($.33*(z)+.33*(p2)+.33*(q2)$) {$T_2$};
          
      \filldraw[lattice] (u1) circle (1.5pt);
      \filldraw[lattice] (v1) circle (1.5pt);
      \filldraw[lattice] (z) circle (1.5pt) node[below right] {$z$};

      \foreach \x in {-2,...,3}{
        \foreach \y in {-2,...,3}{
          \filldraw[lattice] (\x,\y) circle (0.4pt);
        }
      }
      \filldraw[boundary] (p1) circle (1.5pt) node[right] {$p_1$};
      \filldraw[boundary] (q1) circle (1.5pt);
      \filldraw[boundary] (p2) circle (1.5pt) node[below] {$p_2$};
      \filldraw[boundary] (q2) circle (1.5pt);

      \filldraw[lattice] (-1,1) circle (1.5pt) node[above left]{$q$};
      \filldraw[lattice] (r) circle (1.5pt) node[below] {$r$};

      \node at ($.5*(p1)+.5*(q1)$) [above right=-4pt] {$s_{1}$};
      \node at ($.5*(p2)+.5*(q2)$) [below left=-4pt] {$s_{2}$};

      \node at (1.8,-1.8) [above]{$P'$};
    \end{tikzpicture}
    \qquad
    \begin{tikzpicture}
    \coordinate (z) at (0,0);
    \coordinate (p1) at (1,0);
    \coordinate (q1) at (0,1);
    \coordinate (p2) at (0,-1);
    \coordinate (q2) at (-1,-1);
    \coordinate (w) at (2,-1);

    \draw[inequality] (-2.5,3.5) -- (3,-2); 
    \draw[inequality] (-2.5,-1) -- (3,-1); 

    \filldraw[triangle] (z) -- (p1) -- (q1) -- cycle;
    \filldraw[triangle] (z) -- (p2) -- (q2) -- cycle;
    
    \node at ($.33*(z)+.33*(p1)+.33*(q1)$) {$T_1$};
    \node at ($.33*(z)+.33*(p2)+.33*(q2)$) {$T_2$};
    
    \filldraw[lattice] (z) circle (1.5pt) node[below right] {$z$};

    \draw[dashed] (1,0) -- (0,-1);
    
    \foreach \x in {-2,...,3}{
    	\foreach \y in {-2,...,3}{
    		\filldraw[lattice] (\x,\y) circle (0.4pt);
    	}
    }
    \filldraw[boundary] (p1) circle (1.5pt) node[right] {$p_1$};
    \filldraw[boundary] (q1) circle (1.5pt);
    \filldraw[boundary] (p2) circle (1.5pt) node[right,below] {$p_2$};
    \filldraw[boundary] (q2) circle (1.5pt);
    \filldraw[intersect] (w) circle (1.5pt) node[below] {$w$};
    
    \node at ($.5*(p1)+.5*(q1)$) [above right=-4pt] {$s_{1}$};
    \node at ($.5*(p2)+.5*(q2)$) [below left=-4pt] {$s_{2}$};
    
    \node at (-1.5,0.2) [above]{$P'$};
    \end{tikzpicture}
    \caption{Two possibilities for $S_{1}$ and $S_{2}$, which are ruled out a posteriori in the proof of Lemma~\ref{lem:heavy}. Left: $S_{1}$ and $S_{2}$ are parallel. Right: $S_{1}$ and $S_{2}$ intersect at a point.}
    \label{fig:splits-parallelintersect}
  \end{figure}
  We consider the lattice points $q=(-1,1)$ and $r=(1,-1)$.
  As $P'$ has positive genus, it follows from Lemma~\ref{lem:abz} that either $q$ or $r$ is an interior lattice point of $P'$.
  By symmetry we may assume $r\in\interior P'$.
  Then, the point $q$ can either lie on the boundary of $P$ or is not in $P$.
  Now the interval $[(1,1),(1,-1)]$ is the convex hull of two interior lattice points of $P$, but it contains the boundary point $p_1$ in its relative interior.
  Hence we obtain a contradiction.
  
  We conclude that the split lines $S_1$ and $S_2$ intersect at some point $w=(\psi,\omega)$, as depicted in the Figure \ref{fig:splits-parallelintersect}.
  Exploiting the symmetry along the line $x=y$, we may assume that $\omega \leq 0$.
  We use the labels from Figure~\ref{fig:splits-parallelintersect}. and intend to show that $w$ lies in $\partial P$.
  Suppose the contrary.
  Then, due to Lemma~\ref{lem:compatible}, the point $w$ must lie outside $P$.
  From our choice of $w$ we see that point $p_1:=(1,0)$ is closer to $w$ than $(0,1)$.
  From \eqref{eq:det} it follows that $p_2:=(\gamma,\delta)$ is closer to $w$ than $(\alpha,\beta)$.
  Applying Lemma~\ref{lem:area} we infer that the interior of the quadrangle $\conv\{z,p_1,w,p_2\}$ does not contain an interior lattice point of $P$.
  We realize that since the triangles $T_{1}$ and $T_{2}$ do not share an edge, and since $p_1,p_2\in\partial P$ therefore in this case $[p_1,p_2]$ is a boundary edge and the triangle $\conv\{z,p_{1},p_{2}\}$ belongs to $\Delta$.

  We now look at possible coordinates of the point $p_{2}=(\gamma ,\delta)$.
  Since the triangle $\conv\{z,p_1,p_2\}$ is unimodular and as $\delta\leq 0$ we have $\delta=-1$.
  Hence $p_{2}$ lies on the ray $\smallSetOf{(\gamma, \delta)}{\gamma \leq 0, \delta = -1}$.
  We consider some values of $\gamma$ starting with the case of $\gamma = 0$.
  Then \eqref{eq:det} implies that $(\alpha,\beta)$ is on the line $x=-1$, i.e., $\alpha = -1$.
  We explore the possible values for $\beta$.

  For $\beta=-2$ the split edge on the line $S_2$ would be $s_{2}=[(0,-1),(-1,-2)]$.
  This cannot be as then $S_2$ would contain $p_1$ as a third boundary point, a contradiction to Lemma~\ref{lem:cut-split}.
  For $\beta=-1$ and $s_{2}=[(0,-1),(-1,-1)]$, we realize that the polygon $P_{2}$ is bounded in the rectangular strip between the parallel lines $y=x$ and $y=x-1$.
  Yet there is no interior lattice point between them, and this contradicts $G_2$ to have positive genus.
  Any other value of $\beta < -2$ contradicts the convexity of $P$ at $z$. 
  This rules out the possibility $\gamma=0$.

  The arguments for excluding the other values of $\gamma$ are very similar.
  We summarize them briefly.
  For $\gamma = -1$ we obtain $\beta = \alpha + 1$.
  Then either $s_{2}=[(-1,-1),(-1,0)]$, and we obtain $\omega = 2  \geq 0$, which is absurd.
  Or $s_{2}=[(-1,-1),(-2,-1)]$, whence $P_2$ is squeezed between the lines $2y=x$ and $2y=x-1$, which does not leave space for an interior lattice point.
  Or $s_{2}=[(-1,-1),(-3,-2)]$, which then lies on the line spanned by the boundary edge $[p_{1},p_{2}]$. Again, any other value of $\beta$ on the line $y = x + 1$, contradicts the convexity of $P$ at $z$.

  Similarly, when $\gamma = -2$, we obtain $2\beta = \alpha+1$.
  Then either $s_{2}=[(-2,-1),(-1,0)]$, and we obtain $\omega = 1 \geq 0$, which is absurd.
  Or $s_{2}=[(-2,-1),(-3,-1)]$, whence $P_2$ is squeezed between the lines $3y=x$ and $3y=x-1$, which does not leave space for an interior lattice point.
  Or $s_{2}=[(-2,-1),(-5,-2)]$, which then lies on the line spanned by the boundary edge $[p_{1},p_{2}]$.
  Again, any other value of $\beta$ on the line $2y = x + 1$, contradicts the convexity of $P$ at $z$.

  The case $\gamma\leq -3$ is left.
  Then the point $(\alpha,\beta)$ lies on the line $-\gamma y = x + 1$.
  The following three possibilities occur for the split edge $s_{2}$.
  Either $s_{2} = [(-1,0),(\gamma,-1)]$, and we obtain $\omega \geq 0$, forcing $\omega=0$.
  In this case we see that there is no edge to join between the points $(-1,0)$ and $(0,1)$ such that the polygon $P$ remains convex, which gives us a contradiction.
  Or $s_{2} = [(\gamma-1,-1)(\gamma,-1)]$, whence $P_2$ is squeezed between the lines $(1-\gamma)y=x$ and $(1-\gamma)y=x-1$, which does not leave space for an interior lattice point.
  Or $s_{2}=[(2\gamma-1,-2),(\gamma,-1)]$, which then lies on the line spanned by the boundary edge $[p_{1},p_{2}]$.
  Any other value of $\beta$ on the line $-\gamma y = x + 1$, contradicts the convexity of $P$ at $z$.
  
  Therefore, finally, we conclude that there is no lattice point $p_{2} = (\gamma, \delta)$ such that the triangle $\conv\{z,p_{1},p_{2}\}$ belongs to $\Delta$.
  Hence, our initial assumption was wrong, and the triangles $T_{1}$ and $T_{2}$ do share the edge $[z,w]$, where $w=p_1=p_2$ is the intersection of $S_1$ and $S_2$.
  Consequently, $s_1$ and $s_2$ are two distinct edges of $P'$, and so $w$ is a vertex of $P'$.
\end{proof}

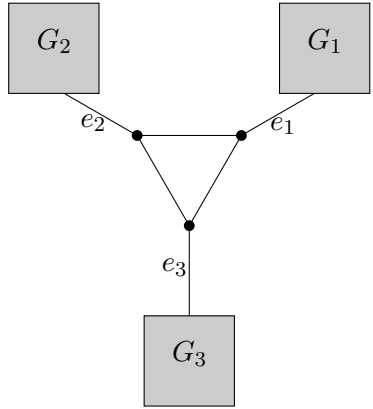
\begin{figure}\centering
	\begin{tikzpicture}[scale=0.8]
	\draw[] (30:1) -- (150:1) -- (270:1) -- cycle;
	
	\draw[] (30:1) -- (30:3);
	\draw[] (150:1) -- (150:3);
	\draw[] (270:1) -- (270:3);
	
	\draw[fill = gray!40!white, draw=black] (1.5,1.2) rectangle ++(1.5,1.5);
	\draw[fill = gray!40!white, draw=black] (-3,1.2) rectangle ++(1.5,1.5);
	\draw[fill = gray!40!white, draw=black] (-0.75,-4) rectangle ++(1.5,1.5);
	
	\fill[black] (30:1) circle (.09cm) node[align=left,   above]{};
	\fill[black] (150:1) circle (.09cm) node[align=left,   above]{};
	\fill[black] (270:1) circle (.09cm) node[align=left,   above]{};
	
	\fill[black] (-1.35,0.4) circle (.0001cm) node[align=left,   above]{$e_{2}\quad$};
	\fill[black] (0,-2) circle (.0001cm) node[align=left,   above]{$e_{3}\quad$};
	\fill[black] (1.8,0.35) circle (.0001cm) node[align=left,   above]{$e_{1}\quad$};
	\fill[black] (0.25,-3.5) circle (.0001cm) node[align=left,   above]{$G_{3}\quad$};
	\fill[black] (-2,1.7) circle (.0001cm) node[align=left,   above]{$G_{2}\quad$};
	\fill[black] (2.5,1.7) circle (.0001cm) node[align=left,   above]{$G_{1}\quad$};
	\end{tikzpicture}
	\caption{A graph with a sprawling triangle}
	\label{fig:sprawling}
\end{figure}

\begin{definition}\label{def:sprawling}
  We say that the planar graph $G$ has a \emph{sprawling triangle} if it is of the form in Figure \ref{fig:sprawling},	
  where the edges $e_{1}$, $e_{2}$ and $e_{3}$ represent cut edges and the shaded regions represent subgraphs with positive genus.
\end{definition}

Note that a graph with a sprawling triangle is a special case of a graph with a heavy cycle; cf.\ Definition~\ref{def:heavy}.
An example arises from the anti-honeycomb triangulation of genus~4 in Example~\ref{exmp:anti-honey-2} and Figure~\ref{fig:antihoney}, and this is characterized by the following.

\begin{theorem}\label{thm:sprawling}
  If $G$ has a sprawling triangle then $g=4$, and, up to unimodular equivalence, we have
  \[
    P \ = \ A_{(-2,0;-2,0,-2,0)} \quad  \text{and} \quad \Delta \ = \ \Delta_{(-2,0;-2,0,-2,0)} \enspace .
  \]
\end{theorem}
\begin{proof}
  We use the notation from Figure~\ref{fig:sprawling}.
  Let $z$ be the central interior lattice point corresponding to the sprawling triangle. Let $s_{1},s_{2}$ and $s_{3}$ be the splits corresponding to the cut edges $e_{1}, e_{2}$ and $e_{3}$.
  Using Lemma \ref{lem:heavy} on these three splits, we obtain that each pair of them intersects at a point on $\partial P$.
  Therefore, the dual of the point $z$ in $\Delta$ is a triangle, too, whose vertices we call $a,b,c$.
  It follows that the lattice triangle $abc$ has normalized area three, and its unimodular triangulation into $abz$, $acz$, and $bcz$ forms an induced subcomplex of $\Delta$.

  As in the proof of Lemma~\ref{lem:abz} we may assume that $a=(1,0)$, $b=(0,1)$ and $z=(0,0)$.
  It then follows that $c=(-1,-1)$.
  Directly from Definition~\ref{def:sprawling} it follows that the line $ab$ is dual to a cut edge in $G$ and thus induces a split of~$\Delta$.
  All the interior lattice points corresponding to a region of the subgraph $G_1$ must lie in those halfspace defined by $ab$ which does not contain~$z$.
  Since $z$ is the only interior lattice point of $P$ which does not correspond to a region of $G_1\cup G_2\cup G_3$, the lattice points $a$, $b$ and $c$ must lie in the boundary of~$P$.

  Now Definition~\ref{def:sprawling} requires the subgraphs $G_1$, $G_2$ and $G_3$ to have positive genus.
  This excludes the possibility that one of the three points $a,b,c$ is a vertex of $P$.
  Thus they must lie on pairwise distinct edges of $P$.
  Now we can apply Lemma~\ref{lem:abz} three times to learn that the three lattice points
  \[
    \begin{aligned}
      (1,0)+(0,1)-(0,0) \ &= \ (1,1) \\ (1,0)+(-1,-1)-(0,0) \ &= \ (0,-1) \\ (0,1)+(-1,-1)-(0,0) \ &= \ (-1,0)
    \end{aligned}
  \]
  are contained in $P$.

  Let $L, M, N$ be the three lines spanned by the edges of $P$ through $a,b,c$, respectively.
  Suppose that $(1,1)$ is a boundary point of $P$.
  Then, since $P$ is lattice polygon, $(1,1)$ is the intersection point of $L$ and $M$ and thus a vertex of $P$.
  This contradicts $G_1$ to have positive genus, and it follows that $(1,1)$ is an interior lattice point of $P$.
  By symmetry, also $(0,-1)$ and $(-1,0)$ are interior lattice points of $P$.

  Now the lattice point $(1,2)$ is not contained in $P$ because it would then need to lie on the line $M$, together with $(-1,0)$.
  But this cannot happen as $(-1,0)$ was already identified as an interior point of $P$.
  Similarly, $(2,1)$ is not contained in $P$ either.
  Yet this implies that the intersection of the lines $L$ and $M$ lies strictly between the two parallel line $(1,0)+\lambda(1,1)$ and $(0,1)+\lambda(1,1)$.
  Thus there is at least one vertex of $P$, which must be a lattice point, on the line $(0,0)+\lambda(1,1)$.
  This forces that $(2,2)$ lies in $P$.

  Again the situation is symmetric, which is why also $(-2,0)$ and $(0,-2)$ lie in $P$.
  The single choice left is that $(2,2)$ as well as $(-2,0)$ and $(0,-2)$ are vertices of $P$, and $L,M,N$ are the only facets.
  This establishes $P = \conv \{(-2,0),(0,-2),(2,2)\}$.
  The claim about $\Delta$ follows as there is only one way to extend the triangulation from the triangle $abc$ to all of $P$.
\end{proof}

\begin{remark}
  Contracting a sprawling triangle in a planar graph yields a graph with a sprawling node, and this cannot be tropically planar.
  Therefore, Example~\ref{exmp:anti-honey-2} shows that the class of tropically planar graphs is not minor closed.
\end{remark}

\section{Graph with a heavy cycle and two loops}

Graphs with a sprawling triangle form special cases of graphs with a heavy cycle.
Thus our main technical result, Lemma~\ref{lem:heavy}, allowed us to derive decisive structural constraints for triangulations whose skeleton has a sprawling triangle in Theorem~\ref{thm:sprawling}.
Now we are looking into another special class of groups with a heavy cycle, aiming for a second structural result on unimodular triangulations of lattice polygons.

\begin{definition}
  We say that a connected trivalent planar graph $G$ has a \emph{heavy cycle with two loops} if it has the form as described in Figure \ref{fig:heavy-two-loops}, where the shaded region represents a subgraph of positive genus.	
\end{definition} 
 
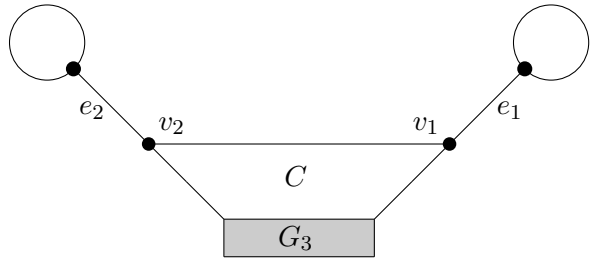
\begin{figure}[ht]\centering
  \begin{tikzpicture}
    \coordinate (v2) at (-0.5,0.5);
	\coordinate (v1) at (3.5,0.5);
	\coordinate (u2) at (-1.5,1.5);
	\coordinate (u1) at (4.5,1.5);
	
	\fill[gray!40!white] (0.5,-0.5) rectangle (2.5,-1);
	\draw[] (0,0) -- (-0.5,0.5);
	\draw[] (-0.5,0.5) -- (3.5,0.5);
	\draw[] (0,0) -- (0.5,-0.5);
	\draw[] (-1.5,1.5) -- (-0.5,0.5);
	\draw[] (3.5,0.5) -- (4.5,1.5);
	\draw[] (-1.85,1.85) circle (0.5cm);
	\draw[] (0.5,-0.5) -- (2.5,-0.5);
	\draw[] (0.5,-0.5) -- (0.5,-1);
	\draw[] (2.5,-0.5) -- (2.5,-1);
	\draw[] (2.5,-1) -- (0.5,-1);
	\draw[] (2.5,-0.5) -- (3,0);
	\draw[] (3,0) -- (3.5,0.5);
	\draw[] (4.85,1.85) circle (0.5cm);
	
    \node at ($(v2)$) [above right=-0.1pt] {$v_{2}$};
	\node at ($(v1)$) [above left=-0.1pt] {$v_{1}$};
	
	\fill[black] (-0.5,0.5) circle (.09cm);
	\fill[black] (3.5,0.5) circle (.09cm);
	\fill[black] (-1.25,0.7) circle (.0001cm) node[align=left,   above]{$e_{2}$};
	\fill[black] (4.3,0.7) circle (.0001cm) node[align=left,   above]{$e_{1}$};
	\fill[black] (4.5,1.5) circle (.1cm) node[align=left,   above]{};
	\fill[black] (-1.5,1.5) circle (.1cm) node[align=left,   above]{};
	\fill[black] (1.65,-0.2) circle (.0001cm) node[align=left,   above]{$C\quad$};
	\fill[black] (1.65,-1.05) circle (.0001cm) node[align=left,   above]{$G_{3}\quad$};
	\end{tikzpicture}
	\caption{Heavy cycle with two loops}
	\label{fig:heavy-two-loops}
\end{figure}

This type of skeleton does actually occur.
\begin{example}\label{exmp:heavy-two-loops}
  The quadrangle $Q_4^{(5)}$ of genus 5, cf.\ \cite[Figure~22]{JB15}, admits a (regular) unimodular triangulation whose skeleton features a heavy cycle with two loops, cf.\ \cite[Figure~23]{JB15}.
\end{example}

Our aim in this section is to establish the following,

\begin{theorem}\label{thm:heavy-two-loops}
  Suppose $G$ is a graph with a heavy cycle $C$ and two loops with cut edges, $e_{1}$ and $e_{2}$, as in Figure \ref{fig:heavy-two-loops}.
  Then the heavy component $P'$ can have at most three interior lattice points, and these lie on the line spanned by the edge $[z,w] \in \Delta$, where $z$ is the interior lattice point dual to $C$, and $w$ is the intersection point of the split edges $s_{1}$ and $s_{2}$.
  In particular, $P'$ is hyperelliptic and $g\leq 5$.
\end{theorem}

\begin{proof}
  It follows from Lemma~\ref{lem:heavy} that the two triangles share the edge $[z,w]$, where $w\in\partial P$ is the point where the two split edges meet.
  We will first show that the interior lattice points of $P'$ lie on the line spanned by $[z,w]$.
	
  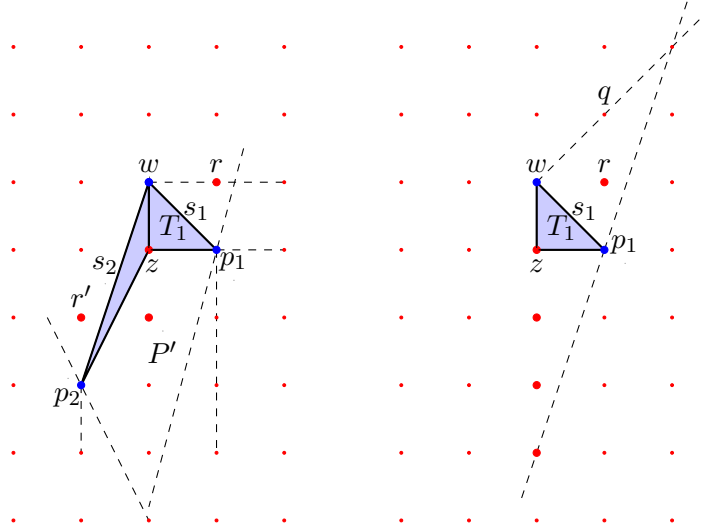
\begin{figure}\centering
    \begin{tikzpicture}[scale=0.9]
      \coordinate (z) at (0,0);
      \coordinate (p1) at (1,0);
      \coordinate (q1) at (0,1);
      \coordinate (p2) at (0,1);
      \coordinate (q2) at (-1,-2);
      \coordinate (u1) at (1,1);
      \coordinate (v1) at (-1,-1);
      \coordinate (k1) at (0,-1);
      
      \filldraw[triangle] (z) -- (p1) -- (q1) -- cycle;
      \filldraw[triangle] (z) -- (p2) -- (q2) -- cycle;
      
      \node[above right] at (z) {$T_1$};
      \filldraw[lattice] (z) circle (1.5pt);
      \draw[] (0,0) -- (0,1);
      \draw[] (0,0) -- (1,0);
      \draw[] (0,1) -- (1,0);
      \draw[] (0,1) -- (-1,-2);
      \draw[] (0,0) -- (-1,-2);
      \draw[dashed] (0,1) -- (2,1);
      \draw[dashed] (1,0) -- (2,0);
      \draw[dashed] (-1,-2) -- (-1,-3);
      \draw[dashed] (1,0) -- (1,-3);
      \draw[dashed] (1.4,1.5) -- (-0,-3.8);
      \draw[dashed] (-1.5,-1) -- (0,-4);
      \fill[black] (0.05,0) circle (.000005cm) node[align=right,   below]{$z$};
      \fill[black] (0,1) circle (.000005cm) node[align=right,   above]{$w$};
      \fill[black] (1,1) circle (.000005cm) node[align=right,   above]{$r$};
      \fill[black] (0.7,0.3) circle (.000005cm) node[align=right,   above]{$s_{1}$};
      \fill[black] (-0.65,-0.5) circle (.000005cm) node[align=right,   above]{$s_{2}$};
      \fill[black] (-1,-1) circle (.000005cm) node[align=right,   above]{$r'$};
      \fill[black] (0.2,-1.2) circle (.000005cm) node[align=right,   below]{$P'$};
      \fill[black] (-1.2,-1.9) circle (.000005cm) node[align=left,   below]{$p_2$};
      \fill[black] (1.25,-0.45) circle (.0000005cm) node[align=left,   above]{$p_1$};
      \foreach \x in {-2,...,2}{
        \foreach \y in {-4,...,3}{
          \filldraw[lattice] (\x,\y) circle (0.6pt);
        }
      }
      \filldraw[boundary] (p1) circle (1.5pt);
      \filldraw[boundary] (q1) circle (1.5pt);
      \filldraw[boundary] (p2) circle (1.5pt);
      \filldraw[boundary] (q2) circle (1.5pt);
      \filldraw[lattice] (u1) circle (1.5pt);
      \filldraw[lattice] (v1) circle (1.5pt);
      \filldraw[lattice] (k1) circle (1.5pt);
    \end{tikzpicture}
    \hspace{1.25cm} 
    \begin{tikzpicture}[scale=0.9]
    \coordinate (z) at (0,0);
    \coordinate (p1) at (1,0);
    \coordinate (q1) at (0,1);
    \coordinate (p2) at (-1,-5);
    \coordinate (q2) at (0,1);
    \coordinate (u1) at (1,1);
    \coordinate (v1) at (-1,-4);
    \coordinate (k1) at (0,-1);
    \coordinate (k2) at (0,-2);
    \coordinate (k3) at (0,-3);
    
    \filldraw[triangle] (z) -- (p1) -- (q1) -- cycle;
    
    \node[above right] at (z) {$T_1$};
    \filldraw[lattice] (z) circle (1.5pt);
    \draw[] (0,0) -- (0,1);
    \draw[] (0,0) -- (1,0);
    \draw[] (0,1) -- (1,0);
    \draw[dashed] (1,0) -- (-0.25,-3.75);
    \draw[dashed] (1,0) -- (2.25,3.75);
    \draw[dashed] (0,1) -- (2.5,3.5);
    \fill[black] (0,0) circle (.000005cm) node[align=right,   below]{$z$};
    \fill[black] (0,1) circle (.000005cm) node[align=right,   above]{$w$};
    \fill[black] (1,1) circle (.000005cm) node[align=right,   above]{$r$};
    \fill[black] (1,2) circle (.000005cm) node[align=right,   above]{$q$};
    \fill[black] (0.7,0.3) circle (.000005cm) node[align=right,   above]{$s_{1}$};
    \fill[black] (1.3,-0.2) circle (.0000005cm) node[align=left,   above]{$p_1$};
    \foreach \x in {-2,...,2}{
    	\foreach \y in {-4,...,3}{
    		\filldraw[lattice] (\x,\y) circle (0.6pt);
    	}
    }
    \filldraw[boundary] (p1) circle (1.5pt);
    \filldraw[boundary] (q1) circle (1.5pt);
    \filldraw[lattice] (u1) circle (1.5pt);
    \filldraw[lattice] (k1) circle (1.5pt);
    \filldraw[lattice] (k2) circle (1.5pt);
    \filldraw[lattice] (k3) circle (1.5pt);
    \end{tikzpicture}
    \caption{This illustrates Theorem \ref{thm:heavy-two-loops}: general sketch (left) and the case when $g(P') \geq 4$ (right), which is impossible}
    \label{fig:cycle_two_loops}
  \end{figure}
  
  As previously, we fix $T_{1} = \conv\{(0,0),(0,1),(1,0)\}$ and use the labels from Figure \ref{fig:cycle_two_loops}; in particular, we may assume that $w=(0,1)$.
  Using Lemma~\ref{lem:heavy} and the unimodularity of $T_{2}$, we realize that $p_{2}$ must lie on the line $x=-1$; so let $p_{2}=(-1,-k)$ for some integer $k$.
  We use Lemma \ref{lem:abz} on the points $z=(0,0)$, $w=(0,1)$ and $p_{1}=(1,0)$ and infer that the point $r:=(1,1)$ is an interior lattice point of $P$; moreover it is the unique interior lattice point of $P_{1}$.
  Similarly, we infer that the point $r':=(-1,-k+1)$ is the unique interior lattice point of $P_{2}$.
  By considering the lines connecting the interior point $r$ with the boundary points $p_1$ and $w$, respectively, we see that $k\geq 0$.
  The same argument shows that the entire polygon $P'$ is squeezed between the lines $x=1$ and $x=-1$.
  In particular, the interior lattice points of $P'$ lie on $x=0$, which is the line spanned by $z$ and $w$.

  Next we will show that $g(P')\leq 3$.
  We assume the contrary, i.e., $g(P')\geq 4$
  Then, since all interior lattice points of $P'$ lie on the line $x=0$, the point $(0,-3)$ must be an interior lattice point of $P$.

  The vertical line $x=1$ contains the point $p_1$, which is a boundary point, and $r$, which is an interior lattice point.
  It follows that there is a lattice point $(1,\lambda)$ in the boundary of $P_1$ for $\lambda>1$; in particular, either $(1,2)\in P_1$ or the boundary edge at $w$ passes through a point in the open interval $((1,2),(1,1))$. Also, as $(0,-3)$ is an interior point, no point in  $\partial P_{1}$ is present on the line $y=3x-3$.
  We realize that in this case $P_{1}$ is contained in the triangle $\conv\{p_1,w,(1,3)\}$.
  However, this triangle has no valid lattice point which could be a vertex of $P_{1}$; recall that $(1,3)$ has been excluded; see Figure~\ref{fig:cycle_two_loops}.
  This provides the desired contradiction, and thus $g(P') \leq 3$.
\end{proof}

The above result is sharp as Example~\ref{exmp:heavy-two-loops} shows.
The following summarizes the known obstructions to tropical planarity together with our new results.

\pagebreak

\begin{theorem}\label{thm:obstructions}
  A trivalent planar graph of genus $g\geq 3$ is not tropically planar if one of the following holds:
  \begin{enumerate}
  \item it contains a sprawling node, or
  \item it contains a sprawling triangle and $g\geq 5$, or
  \item it is crowded, or
  \item it is a TIE-fighter, or
  \item it has a heavy cycle with two loops such that the interior lattice points of the heavy component do not align with the intersection of the two split lines.
  \end{enumerate}
\end{theorem}

\begin{figure}[th]\centering
  \begin{tikzpicture}[scale=0.25]
    \draw[] (-1.5,1.5) -- (0.5,-0.5);
    \draw[] (-1.5,1.5) -- (-3.5,3.5);
    \draw[] (-1.5,1.5) -- (0.5,3.5);
    \draw[] (0.5,-0.5) -- (2.5,-0.5);
    \draw[] (0.5,-0.5) -- (1.5,-2.23);
    \draw[] (1.5,-2.23) -- (2.5,-0.5);
    \draw[] (1.5,-2.23) -- (1.5,-4);
    \draw[] (2.5,-0.5) -- (4.5,1.5);
    \draw[] (6.5,3.5) -- (4.5,1.5);
    \draw[] (4.5,1.5) -- (2.5,3.5);
    \draw[] (1.5,-4) -- (3,-5.5);
    \draw[] (1.5,-4) -- (0,-5.5);
    
    \draw (1.5,-5.5) ellipse (1.5cm and 0.2cm);
    \draw (-1.5,3.5) ellipse (2cm and 0.2cm);
    \draw (4.5,3.5) ellipse (2cm and 0.2cm);
    
    \fill[black] (0.5,-0.5) circle (.09cm);
    \fill[black] (2.5,-0.5) circle (.09cm);
    \fill[black] (1.5,-2.23) circle (.09cm);
    \fill[black] (-1.5,1.5) circle (.09cm);
    \fill[black] (4.5,1.5) circle (.09cm);
    \fill[black] (1.5,-4) circle (.09cm);
    \fill[black] (-3.5,3.5) circle (.09cm);
    \fill[black] (0.5,3.5) circle (.09cm);
    \fill[black] (2.5,3.5) circle (.09cm);
    \fill[black] (6.5,3.5) circle (.09cm);
    \fill[black] (0,-5.5) circle (.09cm);
    \fill[black] (3,-5.5) circle (.09cm);
  \end{tikzpicture}
  \qquad
  \begin{tikzpicture}[scale=0.3]
    \draw[] (-2,0) -- (2,0);
    \draw[] (-2,0) -- (-3,1);
    \draw[] (-2,0) -- (-2,-1);
    \draw[] (2,0) -- (3,1);
    \draw[] (2,0) -- (2,-1);
    \draw[] (-2,-3) -- (-2,-4);
    \draw[] (2,-3) -- (2,-4);
    \draw[] (-2,-6) -- (2,-6);
    
    \draw (-3.35,1.35) circle (0.5cm);
    \draw (3.35,1.35) circle (0.5cm);
    \draw (-2,-2) ellipse (0.2cm and 1cm);
    \draw (-2,-5) ellipse (0.2cm and 1cm);
    \draw (2,-2) ellipse (0.2cm and 1cm);
    \draw (2,-5) ellipse (0.2cm and 1cm);
    
    \fill[black] (-2,0) circle (.09cm);
    \fill[black] (-2,-1) circle (.09cm);
    \fill[black] (-2,-3) circle (.09cm);
    \fill[black] (-2,-4) circle (.09cm);
    \fill[black] (-2,-6) circle (.09cm);
    \fill[black] (2,0) circle (.09cm);
    \fill[black] (2,-1) circle (.09cm);
    \fill[black] (2,-3) circle (.09cm);
    \fill[black] (2,-4) circle (.09cm);
    \fill[black] (2,-6) circle (.09cm);
    \fill[black] (-3,1) circle (.09cm);
    \fill[black] (3,1) circle (.09cm);
  \end{tikzpicture}
  \qquad
  \begin{tikzpicture}[scale=0.2]
    \draw[] (-1.5,1.5) -- (0.5,-0.5);
    \draw[] (-1.5,1.5) -- (-3.5,3.5);
    \draw[] (-1.5,1.5) -- (0.5,3.5);
    \draw[] (0.5,-0.5) -- (2.5,-0.5);
    \draw[] (0.5,-0.5) -- (1.5,-2.23);
    \draw[] (1.5,-2.23) -- (2.5,-0.5);
    \draw[] (1.5,-2.23) -- (1.5,-4);
    \draw[] (2.5,-0.5) -- (4.5,1.5);
    \draw[] (6.5,3.5) -- (4.5,1.5);
    \draw[] (4.5,1.5) -- (2.5,3.5);
    \draw[] (1.5,-4) -- (3,-5.5);
    \draw[] (1.5,-4) -- (0,-5.5);
    \draw[] (0,-7.5) -- (3,-7.5);
    
    \draw (0,-6.5) ellipse (0.2cm and 1cm);
    \draw (3,-6.5) ellipse (0.2cm and 1cm);
    \draw (-1.5,3.5) ellipse (2cm and 0.2cm);
    \draw (4.5,3.5) ellipse (2cm and 0.2cm);
    
    \fill[black] (0.5,-0.5) circle (.09cm);
    \fill[black] (2.5,-0.5) circle (.09cm);
    \fill[black] (1.5,-2.23) circle (.09cm);
    \fill[black] (-1.5,1.5) circle (.09cm);
    \fill[black] (4.5,1.5) circle (.09cm);
    \fill[black] (1.5,-4) circle (.09cm);
    \fill[black] (-3.5,3.5) circle (.09cm);
    \fill[black] (0.5,3.5) circle (.09cm);
    \fill[black] (2.5,3.5) circle (.09cm);
    \fill[black] (6.5,3.5) circle (.09cm);
    \fill[black] (0,-5.5) circle (.09cm);
    \fill[black] (3,-5.5) circle (.09cm);
    \fill[black] (0,-7.5) circle (.09cm);
    \fill[black] (3,-7.5) circle (.09cm);
  \end{tikzpicture}
  \qquad
  \begin{tikzpicture}[scale=0.3]
    \draw[] (-2,0) -- (2,0);
    \draw[] (-2,0) -- (-3,1);
    \draw[] (-2,0) -- (-2,-1);
    \draw[] (2,0) -- (3,1);
    \draw[] (2,0) -- (2,-1);
    \draw[] (-2,-3) -- (-2,-4);
    \draw[] (2,-3) -- (2,-4);
    \draw[] (-2,-6) -- (-1,-7);
    \draw[] (2,-6) -- (1,-7);
    
    \draw (-3.35,1.35) circle (0.5cm);
    \draw (3.35,1.35) circle (0.5cm);
    \draw (-2,-2) ellipse (0.2cm and 1cm);
    \draw (-2,-5) ellipse (0.2cm and 1cm);
    \draw (2,-2) ellipse (0.2cm and 1cm);
    \draw (2,-5) ellipse (0.2cm and 1cm);
    \draw (0,-7) ellipse (1cm and 0.2cm);
    
    \fill[black] (-2,0) circle (.09cm);
    \fill[black] (-2,-1) circle (.09cm);
    \fill[black] (-2,-3) circle (.09cm);
    \fill[black] (-2,-4) circle (.09cm);
    \fill[black] (-2,-6) circle (.09cm);
    \fill[black] (2,0) circle (.09cm);
    \fill[black] (2,-1) circle (.09cm);
    \fill[black] (2,-3) circle (.09cm);
    \fill[black] (2,-4) circle (.09cm);
    \fill[black] (2,-6) circle (.09cm);
    \fill[black] (-3,1) circle (.09cm);
    \fill[black] (3,1) circle (.09cm);
    \fill[black] (1,-7) circle (.09cm);
    \fill[black] (-1,-7) circle (.09cm);
  \end{tikzpicture}
  \caption{Examples of non realizable graphs of genus 7 and 8}
  \label{fig:g=7and8}
\end{figure}
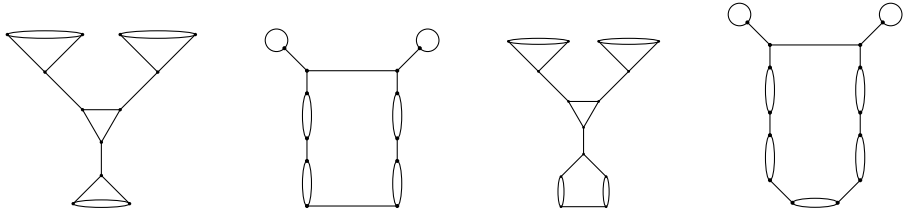

For instance, the preceding result excludes the graphs in Figure \ref{fig:g=7and8}.

\section{Anti-honeycombs}
\label{sec:anti-honey}

The purpose of this section is to define a class of lattice polygons each of which admits a special triangulation.
This is motivated by Theorem~\ref{thm:sprawling}, which characterizes one of these triangulations.
We believe that the entire family deserves some attention.
Consider three families of parallel lines:
\begin{equation}\label{eq:parallel-lines}
  L_k = \{ y {=} 2x {+} k \} \,,\ M_\ell = \{ y {=} x/2 {-} \ell \} \,,\ N_m = \{ y {=} {-} x {+} m \} \enspace ,
\end{equation}
where $k,\ell,m\in\ZZ$.
By picking a sixtuple $\pi=(k,k';\ell,\ell';m,m')$, with $k<k'$, $\ell<\ell'$, and $m<m'$ we obtain a polygon $A_\pi$ which is defined by three pairs of inequalities, where each pair comes from one of the parallel families~\eqref{eq:parallel-lines}.
The polygon $A_\pi$ has integral vertices, and we call it the \emph{anti-honeycomb polygon of type $\pi$}.
The name comes about from the connection to the \enquote{honeycomb polygons} studied in \cite[pp.~10ff]{JB15}.
Note that not all of the six inequalities need to be facet defining, whence $A_\pi$ is a hexagon, a pentagon, a quadrangle or a triangle.
For instance,
\[
  A_{(0,k;0,k;0,k)} \ = \ \conv\{(-k,-k),(0,k),(k,0)\}
\]
is a triangle, and its genus equals
\[  g(A_{(0,k;0,k;0,k)}) \ = \ \frac{3k^{2} -3k +2}{2} \enspace .\]

We fix a type $\pi=(k,k';\ell,\ell';m,m')$, and we let $V=A_{\pi}\cap\ZZ^2$ be the set of lattice points in $A_\pi$.
Intersecting with the shifted lattice
\[
  \cL \ = \
  \begin{pmatrix}
    -1\\
    -1\\
  \end{pmatrix}
  + \ZZ
  \begin{pmatrix}
    2\\
    1\\
  \end{pmatrix}
  + \ZZ
  \begin{pmatrix}
    1\\
    2\\
  \end{pmatrix}
\]
of index~3 we obtain a subset $V'=V\cap\cL$ of the lattice points in $A_\pi$.
The families of lines \eqref{eq:parallel-lines} yield weakly compatible splits of the point configuration~$V'$, and they induce a triangulation $\Delta_\pi'$ of the lattice points in $V'$.
Notice that all the points in $V \setminus V'$ lie in the interior $\interior A_{\pi}$.
Since none of these lattice points lies on any of the lines~\eqref{eq:parallel-lines} it follows that each of them is contained in the interior of a unique triangle of $\Delta_\pi'$.
Employing stellar subdivisions at the points in $V \setminus V'$ this yields a triangulation $\Delta_\pi$ of $V$, which we call the \emph{anti-honeycomb triangulation of type $\pi$}.
Figure~\ref{fig:antihoney}, Example~\ref{exmp:anti-honey-2}, and Theorem~\ref{thm:sprawling} are concerned with the case $\pi=(-2,0;-2,0;-2,0)$ of genus~4.
Figure \ref{fig:antihoney19} shows a genus 19 anti-honeycomb triangulation along with its corresponding skeleton for $\pi=(0,4;0,4;0,4)$.

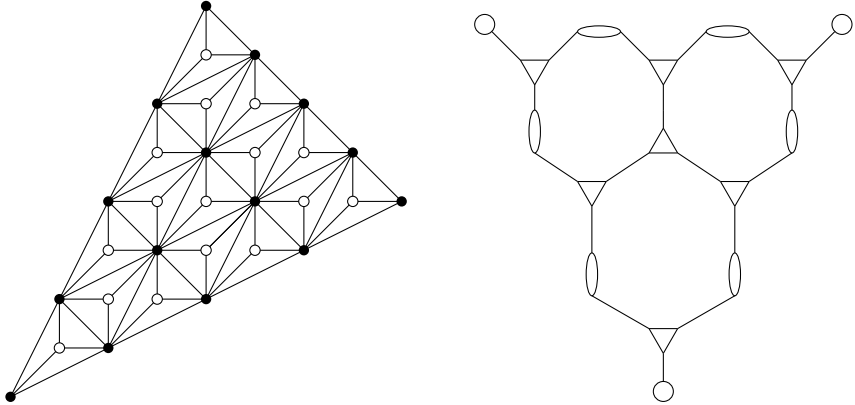
\begin{figure}\centering
  \newcommand\sz{3pt}
  \begin{tikzpicture}[scale=0.65]
    \draw[] (4,0) -- (-4,-4);
    \draw[] (4,0) -- (0,4);
    \draw[] (0,4) -- (-4,-4);
    \draw[] (-4,-4) -- (-3,-3);
    \draw[] (-3,-3) -- (-3,-2);
    \draw[] (-3,-3) -- (-2,-3);
    \draw[] (-3,-2) -- (-2,-3);
    \draw[] (-3,-2) -- (-2,-2);
    \draw[] (-2,-3) -- (-2,-2);
    \draw[] (-2,-2) -- (-1,-1);
    \draw[] (-3,-2) -- (-1,-1);
    \draw[] (-2,-3) -- (-1,-1);
    \draw[] (-2,-1) -- (-1,-1);
    \draw[] (-1,-2) -- (-1,-1);
    \draw[] (-2,-1) -- (-3,-2);
    \draw[] (-1,-2) -- (-2,-3);
    \draw[] (-1,-2) -- (0,-2);
    \draw[] (-2,-1) -- (-2,0);
    \draw[] (-2,-0) -- (-1,-1);
    \draw[] (0,-2) -- (-1,-1);
    \draw[] (-2,-0) -- (0,1);
    \draw[] (0,-2) -- (1,0);
    \draw[] (-2,0) -- (-1,1);
    \draw[] (-1,1) -- (-1,2);
    \draw[] (-1,2) -- (0,1);
    \draw[] (0,1) -- (-1,1);
    \draw[] (1,-1) -- (1,0);
    \draw[] (2,-1) -- (1,0);
    \draw[] (1,-1) -- (0,-2);
    \draw[] (1,-1) -- (2,-1);
    \draw[] (0,1) -- (2,2);
    \draw[] (2,2) -- (1,0);
    \draw[] (1,2) -- (2,2);
    \draw[] (1,2) -- (1,3);
    \draw[] (0,1) -- (1,3);
    \draw[] (0,1) -- (1,2);
    \draw[] (-1,2) -- (1,3);
    \draw[] (1,0) -- (3,1);
    \draw[] (3,1) -- (2,-1);
    \draw[] (0,3) -- (0,4);
    \draw[] (0,3) -- (-1,2);
    \draw[] (0,3) -- (1,3);
    \draw[] (0,2) -- (1,3);
    \draw[] (0,2) -- (-1,2);
    \draw[] (0,2) -- (0,1);
    \draw[] (2,1) -- (2,2);
    \draw[] (2,1) -- (3,1);
    \draw[] (1,0) -- (2,1);
    \draw[] (1,0) -- (2,0);
    \draw[] (2,0) -- (3,1);
    \draw[] (2,0) -- (2,-1);
    \draw[] (3,0) -- (3,1);
    \draw[] (3,0) -- (4,0);
    \draw[] (3,0) -- (2,-1);
    \draw[] (2,2) -- (1,1);
    \draw[] (1,1) -- (1,0);
    \draw[] (1,1) -- (0,1);
    \draw[] (0,1) -- (1,0);
    \draw[] (0,1) -- (0,0);
    \draw[] (1,0) -- (0,0);
    \draw[] (0,1) -- (-1,0);
    \draw[] (0,1) -- (-1,-1);
    \draw[] (1,0) -- (-1,-1);
    \draw[] (1,0) -- (0,-1);
    \draw[] (-1,0) -- (-2,0);
    \draw[] (-1,0) -- (-1,-1);
    \draw[] (1,0) -- (0,-1);
    \draw[] (0,-1) -- (-1,-1);
    \draw[] (0,-1) -- (0,-2);
    \draw[] (0,0) -- (-1,-1);
    \foreach \x/\y in {-3/-3, -2/-2, -2/-1, -1/-2, -1/0, 0/0, -1/1, 0/-1, 0/2, 0/3, 1/-1, 1/1, 1/2, 2/0, 2/1, 3/0}{
      \filldraw[fill=white,draw=black] (\x,\y) circle (\sz);
    }
    \foreach \x/\y in {-4/-4, -3/-2, -2/-3, 0/-2, -2/0, 2/2, 0/1, 1/0, -1/-1, -1/2, 2/-1, 4/0, 0/4, 1/3, 3/1}{
      \fill[black] (\x,\y) circle (\sz);
    }
  \end{tikzpicture}
  \qquad
  \begin{tikzpicture}[scale=0.19]
    \draw[] (0,0) -- (2,0);
    \draw[] (0,0) -- (1,-1.732);
    \draw[] (2,0) -- (1,-1.732);
    \draw[] (2,0) -- (4,2);
    \draw[] (0,0) -- (-2,2);
    \draw[] (1,-1.732) -- (1,-3.5);
    \draw (1,-5) ellipse (0.4cm and 1.5cm);
    \draw[] (1,-6.5) -- (-2,-8.5);
    \draw[] (-2,-8.5) -- (-4,-8.5);
    \draw[] (-2,-8.5) -- (-3,-10.232);
    \draw[] (-4,-8.5) -- (-3,-10.232);
    \draw[] (-3,-13.5) -- (-3,-10.232);
    \draw (-3,-15) ellipse (0.4cm and 1.5cm);
    \draw[] (-3,-16.5) -- (-7,-18.8);
    \draw[] (-4,-8.5) -- (-7,-6.5);
    \draw[] (-7,-6.5) -- (-9,-6.5);
    \draw[] (-7,-6.5) -- (-8,-4.768);
    \draw[] (-9,-6.5) -- (-8,-4.768);
    \draw[] (-8,-3.5) -- (-8,-4.768);
    \draw[] (-9,-6.5) -- (-12,-8.5);
    \draw[] (-12,-8.5) -- (-14,-8.5);
    \draw[] (-12,-8.5) -- (-13,-10.232);
    \draw[] (-14,-8.5) -- (-13,-10.232);
    \draw[] (-13,-13.5) -- (-13,-10.232);
    \draw (-13,-15) ellipse (0.4cm and 1.5cm);
    \draw[] (-13,-16.5) -- (-9,-18.8);
    \draw[] (-7,-18.8) -- (-9,-18.8);
    \draw[] (-8,-20.532) -- (-7,-18.8);
    \draw[] (-8,-20.532) -- (-9,-18.8);
    \draw[] (-8,-20.532) -- (-8,-22.5);
    \draw (-8,-23.2) circle (0.7cm);
    \draw[] (-14,-8.5) -- (-17,-6.5);
    \draw (-17,-5) ellipse (0.4cm and 1.5cm);
    \draw (4.5,2.5) circle (0.7cm);
    \draw (-3.5,2) ellipse (1.5cm and 0.4cm);
    \draw[] (-5,2) -- (-7,0);
    \draw[] (-7,0) -- (-9,0);
    \draw[] (-7,0) -- (-8,-1.732);
    \draw[] (-8,-1.732) -- (-9,0);
    \draw[] (-9,0) -- (-11,2);
    \draw[] (-8,-1.732) -- (-8,-3.5);
    \draw (-12.5,2) ellipse (1.5cm and 0.4cm);
    \draw[] (-14,2) -- (-16,0);
    \draw[] (-16,0) -- (-18,0);
    \draw[] (-16,0) -- (-17,-1.732);
    \draw[] (-17,-1.732) -- (-18,0);
    \draw[] (-18,0) -- (-20,2);
    \draw (-20.5,2.5) circle (0.7cm);
    \draw[] (-17,-1.732) -- (-17,-3.5);	
  \end{tikzpicture}
  \caption{Anti-honeycomb triangulation of genus 19 on the left, and the corresponding skeleton on the right}
  \label{fig:antihoney19}
\end{figure}

The honeycomb curves yield moduli cones of maximal dimension $2g+1$, where $g$ is the genus, cf.\ \cite[Theorem~1]{JB15}.
In contrast the anti-honeycomb curves form a large family whose moduli cones are much smaller.
For instance, a direct \polymake \cite{DMV:polymake} computation shows that the moduli cone of $\Delta_{(0,4;0,4;0,4)}$ is only 28-dimensional, whereas the upper bound $2g+1$ equals~39.

Up to an affine transformation, the three families of lines in \eqref{eq:parallel-lines} form a Coxeter hyperplane arrangement of type $\tilde A_2$.
This generalizes to arbitrary dimensions, and so does the construction of the anti-honeycomb triangulations.
The resulting anti-honeycomb polytopes are affine images of the \enquote{alcoved polytopes} of Lam and Postnikov~\cite{LamPostnikov05}.
Further details will be explored elsewhere.

\section{Conclusion}

The classication of the tropically planar graphs of genus $g\leq 5$ was obtained in \cite{JB15}.
Theorem ~\ref{thm:obstructions} now allows for a combinatorial characterization:

\begin{corollary}
  A trivalent planar graph of genus $g\leq 5$ is tropically planar if and only if none of the obstructions in Theorem~\ref{thm:obstructions} occurs.
\end{corollary}

\begin{proof}
  The trivalent graphs of low genus have been classified in \cite{chemicalgraphs}.
  For $g=3$ there are five such graphs, one of which has a sprawling node; the other four are tropically planar \cite[Theorem~5.1]{JB15}.
  For $g=4$ there are 17 graphs: one is non-planar, three have a sprawling node, the remaining 13 are tropically planar \cite[Theorem~7.1]{JB15}.
  This was known before.
  
  There are exactly 71 trivalent graphs of genus 5.
  Among them only 52 are planar without a sprawling node \cite{JB15}.
  Of these 14 were ruled out by explicit computations, which leaves 38 tropically planar graphs of genus 5.
  One of the key contributions of \cite[Figure 8]{M19} is to obtain general obstructions to tropical planarity, which rules out another ten, which are crowded or TIE-fighters.

  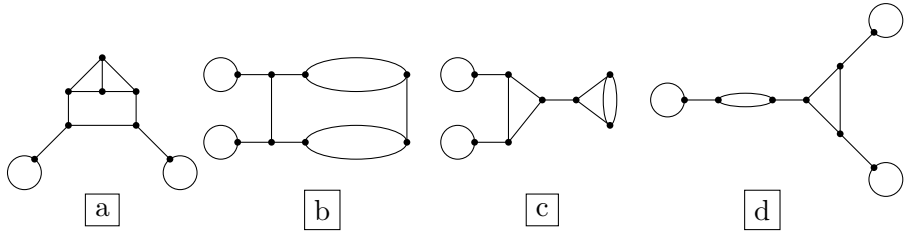
\begin{figure}[th]\centering
    \begin{tikzpicture}[scale=0.45]
      \draw[] (0,0) -- (2,0);
      \draw[] (0,0) -- (0,1);
      \draw[] (2,0) -- (2,1);
      \draw[] (0,1) -- (1,1);
      \draw[] (1,1) -- (2,1);
      \draw[] (1,1) -- (1,2);
      \draw[] (0,1) -- (1,2);
      \draw[] (2,1) -- (1,2);
      \draw[] (0,0) -- (-1,-1);
      \draw[] (2,0) -- (3,-1);
      \draw (-1.3,-1.4) circle (0.5cm);
      \draw (3.3,-1.4) circle (0.5cm);
      \node[draw] at (1,-2.5) {a};
      \fill[black] (0,0) circle (.1cm) node[align=left,   above]{};
      \fill[black] (2,0) circle (.1cm) node[align=left,   above]{};
      \fill[black] (1,1) circle (.1cm) node[align=left,   above]{};
      \fill[black] (0,1) circle (.1cm) node[align=left,   above]{};
      \fill[black] (2,1) circle (.1cm) node[align=left,   above]{};
      \fill[black] (1,2) circle (.1cm) node[align=left,   above]{};
      \fill[black] (-1,-1) circle (.1cm) node[align=left,   above]{};
      \fill[black] (3,-1) circle (.1cm) node[align=left,   above]{};
      \draw[] (6,-0.5) -- (5,-0.5);
      \draw[] (6,-0.5) -- (6,1.5);
      \draw[] (6,1.5) -- (5,1.5);
      \draw[] (6,1.5) -- (7,1.5);
      \draw[] (6,-0.5) -- (7,-0.5);
      \draw (4.5,-0.5) circle (0.5cm);
      \draw (4.5,1.5) circle (0.5cm);
      \draw (8.5,1.5) ellipse (1.5cm and 0.5cm);
      \draw (8.5,-0.5) ellipse (1.5cm and 0.5cm);
      \draw[] (10,-0.5) -- (10,1.5);
      \node[draw] at (7.5,-2.5) {b};
      \fill[black] (6,-0.5) circle (.1cm) node[align=left,   above]{};
      \fill[black] (7,-0.5) circle (.1cm) node[align=left,   above]{};
      \fill[black] (5,-0.5) circle (.1cm) node[align=left,   above]{};
      \fill[black] (6,1.5) circle (.1cm) node[align=left,   above]{};
      \fill[black] (7,1.5) circle (.1cm) node[align=left,   above]{};
      \fill[black] (5,1.5) circle (.1cm) node[align=left,   above]{};
      \fill[black] (10,-0.5) circle (.1cm) node[align=left,   above]{};
      \fill[black] (10,1.5) circle (.1cm) node[align=left,   above]{};
      \draw[] (13,-0.5) -- (12,-0.5);
      \draw[] (13,-0.5) -- (13,1.5);
      \draw[] (13,1.5) -- (12,1.5);
      \draw[] (13,1.5) -- (14,0.75);
      \draw[] (13,-0.5) -- (14,0.75);
      \draw[] (14,0.75) -- (15,0.75);
      \draw[] (15,0.75) -- (16,1.5);
      \draw[] (15,0.75) -- (16,0);
      \draw (11.5,-0.5) circle (0.5cm);
      \draw (11.5,1.5) circle (0.5cm);
      \draw (16,0.75) ellipse (0.2cm and 0.75cm);
      \node[draw] at (14,-2.5) {c};
      \fill[black] (13,-0.5) circle (.1cm) node[align=left,   above]{};
      \fill[black] (12,-0.5) circle (.1cm) node[align=left,   above]{};
      \fill[black] (13,1.5) circle (.1cm) node[align=left,   above]{};
      \fill[black] (12,1.5) circle (.1cm) node[align=left,   above]{};
      \fill[black] (14,0.75) circle (.1cm) node[align=left,   above]{};
      \fill[black] (15,0.75) circle (.1cm) node[align=left,   above]{};
      \fill[black] (16,1.5) circle (.1cm) node[align=left,   above]{};
      \fill[black] (16,0) circle (.1cm) node[align=left,   above]{};
      \draw (20,0.75) ellipse (0.8cm and 0.2cm);
      \draw[] (19.2,0.75) -- (18.2,0.75);
      \draw[] (20.8,0.75) -- (21.8,0.75);
      \draw[] (21.8,0.75) -- (22.8,1.75);
      \draw[] (21.8,0.75) -- (22.8,-0.25);
      \draw[] (22.8,1.75) -- (22.8,-0.25);
      \draw[] (22.8,1.75) -- (23.8,2.75);
      \draw[] (22.8,-0.25) -- (23.8,-1.25);
      \draw (17.7,0.75) circle (0.5cm);
      \draw (24.15,3.10) circle (0.5cm);
      \draw (24.15,-1.6) circle (0.5cm);
      \node[draw] at (20.5,-2.5) {d};
      \fill[black] (19.2,0.75) circle (.1cm) node[align=left,   above]{};
      \fill[black] (20.8,0.75) circle (.1cm) node[align=left,   above]{};
      \fill[black] (21.8,0.75) circle (.1cm) node[align=left,   above]{};
      \fill[black] (18.2,0.75) circle (.1cm) node[align=left,   above]{};
      \fill[black] (22.8,1.75) circle (.1cm) node[align=left,   above]{};
      \fill[black] (22.8,-0.25) circle (.1cm) node[align=left,   above]{};
      \fill[black] (23.8,2.75) circle (.1cm) node[align=left,   above]{};
      \fill[black] (23.8,-1.25) circle (.1cm) node[align=left,   above]{};
    \end{tikzpicture}
    \caption{The four genus 5 graphs that are not ruled out by any prior known criteria}
    \label{fig:g=5_nonrealized}
  \end{figure}

  As our new contribution we can now discuss the remaining four graphs, which are shown in Figure~\ref{fig:g=5_nonrealized}.
  Firstly, we observe that all of these exhibit a heavy cycle.
  The graph labeled \enquote{a} has a heavy cycle with two loops, but the component away from the two loops is not hyperelliptic; i.e., it is ruled out by Theorem~\ref{thm:heavy-two-loops}.
  The second graph, labeled \enquote{b} also has a heavy cycle with two loops, the component (of genus 3) away from the two loops is even hyperelliptic.
  However, the three interior lattice points of the genus~3 component do not lie on a common line with the point of intersection of the two splits.
  Thus \enquote{b} is ruled out by Theorem~\ref{thm:heavy-two-loops}, too.
  The graphs labeled \enquote{c} and \enquote{d} feature sprawling triangles, whence they are ruled out by Theorem~\ref{thm:sprawling}.
  This completes our combinatorial characterization of the tropically planar graphs of genus at most five.
\end{proof}

For genus 6, there are 388 trivalent graphs altogether, 354 of which are planar \cite{chemicalgraphs}.
In \cite{M19} it was shown that 152 tropically planar graphs of genus~6 remain.
There are 28 graphs which are non-realizable and could not be ruled out using any prior known criteria; cf.\ \cite[Figure 17]{M19}.
Out of these 28 graphs, 19 have a heavy cycle with two loops and can be ruled out using Theorem \ref{thm:heavy-two-loops} because the genus is too high.
One of the remaining graphs has a sprawling triangle, and thus excluded by Theorem~\ref{thm:sprawling}.
We are left with eight graphs of genus 6, which are shown in Figure \ref{fig:genus6_uncategorized}; for these we are not aware of any a priori obstruction.

It would be interesting to know how the tropically plane curves of a fixed genus fit into the moduli space of all tropical curves.
For genus~3 this was recently answered in terms of modifications by Hahn et al.~\cite{HahnMarkwigRenTyomkin:2019.02440}.

	\begin{figure}\centering
		\begin{tikzpicture}[scale=0.4]
		\draw[] (0,0) -- (2,0);
		\draw[] (0,0) -- (0,-2);
		\draw[] (0,-2) -- (2,-2);
		\draw[] (2,0) -- (1.5,-1);
		\draw[] (2,0) -- (2.5,-1);
		\draw[] (1.5,-1) -- (2.5,-1);
		\draw[] (1.5,-1) -- (2,-2);
		\draw[] (2.5,-1) -- (2,-2);
		\draw[] (0,0) -- (-1,0);
		\draw[] (-1,0) -- (-2,-1);
		\draw[] (-1,0) -- (-2,1);
		\draw[] (0,-2) -- (-1,-2);
		\draw (-1.5,-2) circle (0.5cm);
		\draw (-2,0) ellipse (0.2cm and 1cm);
		\node[draw] at (0,-3) {a};
		\fill[black] (0,0) circle (.1cm) node[align=left,   above]{};
		\fill[black] (0,-2) circle (.1cm) node[align=left,   above]{};
		\fill[black] (2,0) circle (.1cm) node[align=left,   above]{};
		\fill[black] (-1,0) circle (.1cm) node[align=left,   above]{};
		\fill[black] (-2,-1) circle (.1cm) node[align=left,   above]{};
		\fill[black] (-2,1) circle (.1cm) node[align=left,   above]{};
		\fill[black] (-1,-2) circle (.1cm) node[align=left,   above]{};
		\fill[black] (0,-2) circle (.1cm) node[align=left,   above]{};
		\fill[black] (2,-2) circle (.1cm) node[align=left,   above]{};
		\fill[black] (1.5,-1) circle (.1cm) node[align=left,   above]{};
		\fill[black] (2.5,-1) circle (.1cm) node[align=left,   above]{};
		\end{tikzpicture}
		\qquad
		\begin{tikzpicture}[scale=0.4]
		\draw[] (15,0) -- (16,0);
		\draw[] (13,0) -- (13,-2);
		\draw[] (15,-2) -- (16,-2);
		\draw[] (16,0) -- (16,-2);
		\draw[] (16,-2) -- (17,-2);
		\draw[] (16,0) -- (17,0);
		\draw[] (17,0) -- (18,1);
		\draw[] (17,0) -- (18,-1);
		\draw (17.5,-2) circle (0.5cm);
		\draw (14,0) ellipse (1cm and 0.3cm);
		\draw (14,-2) ellipse (1cm and 0.3cm);
		\draw (18,0) ellipse (0.2cm and 1cm);
		\node[draw] at (16,-3) {b};
		\fill[black] (13,0) circle (.1cm) node[align=left,   above]{};
		\fill[black] (13,-2) circle (.1cm) node[align=left,   above]{};
		\fill[black] (15,0) circle (.1cm) node[align=left,   above]{};
		\fill[black] (15,-2) circle (.1cm) node[align=left,   above]{};
		\fill[black] (16,0) circle (.1cm) node[align=left,   above]{};
		\fill[black] (16,-2) circle (.1cm) node[align=left,   above]{};
		\fill[black] (17,0) circle (.1cm) node[align=left,   above]{};
		\fill[black] (18,1) circle (.1cm) node[align=left,   above]{};
		\fill[black] (18,-1) circle (.1cm) node[align=left,   above]{};
		\fill[black] (17,-2) circle (.1cm) node[align=left,   above]{};
		\end{tikzpicture}
		\qquad
		\vspace{0.3cm}
		\begin{tikzpicture}[scale=0.4]
		\draw[] (0,-4) -- (2,-4);
		\draw[] (0,-4) -- (0,-6);
		\draw[] (0,-6) -- (2,-6);
		\draw[] (0,-4) -- (-2,-4);
		\draw[] (0,-6) -- (-2,-6);
		\draw[] (2,-4) -- (2,-6);
		\draw[] (2,-4) -- (3,-4);
		\draw[] (2,-6) -- (3,-6);
		\draw[] (5,-6) -- (6,-6);
		\draw (3.5,-4) circle (0.5cm);
		\draw (6.5,-6) circle (0.5cm);
		\draw (-2,-5) ellipse (0.2cm and 1cm);
		\draw (4,-6) ellipse (1cm and 0.2cm);
		\node[draw] at (0,-7) {c};
		\fill[black] (0,-4) circle (.1cm) node[align=left,   above]{};
		\fill[black] (0,-6) circle (.1cm) node[align=left,   above]{};
		\fill[black] (-2,-6) circle (.1cm) node[align=left,   above]{};
		\fill[black] (2,-6) circle (.1cm) node[align=left,   above]{};
		\fill[black] (3,-6) circle (.1cm) node[align=left,   above]{};
		\fill[black] (5,-6) circle (.1cm) node[align=left,   above]{};
		\fill[black] (2,-4) circle (.1cm) node[align=left,   above]{};
		\fill[black] (-2,-4) circle (.1cm) node[align=left,   above]{};
		\fill[black] (3,-4) circle (.1cm) node[align=left,   above]{};
		\end{tikzpicture}
		\qquad
		\begin{tikzpicture}[scale=0.4]
		\draw[] (12,-4) -- (12,-6);
		\draw[] (12,-6) -- (11,-6);
		\draw[] (12,-4) -- (13,-4);
		\draw[] (12,-6) -- (13,-6);
		\draw[] (13,-4) -- (13,-6);
		\draw[] (13,-4) -- (14,-5);
		\draw[] (13,-6) -- (14,-5);
		\draw[] (14,-5) -- (15,-5);
		\draw[] (11,-4) -- (12,-4);
		\draw[] (9,-4) -- (8,-4);
		\draw (10,-4) ellipse (1cm and 0.2cm);
		\draw (7.5,-4) circle (0.5cm);
		\draw (10.5,-6) circle (0.5cm);
		\draw (15.5,-5) circle (0.5cm);
		\node[draw] at (12,-7) {d};
		\fill[black] (12,-4) circle (.1cm) node[align=left,   above]{};
		\fill[black] (11,-4) circle (.1cm) node[align=left,   above]{};
		\fill[black] (9,-4) circle (.1cm) node[align=left,   above]{};
		\fill[black] (8,-4) circle (.1cm) node[align=left,   above]{};
		\fill[black] (13,-4) circle (.1cm) node[align=left,   above]{};
		\fill[black] (12,-6) circle (.1cm) node[align=left,   above]{};
		\fill[black] (11,-6) circle (.1cm) node[align=left,   above]{};
		\fill[black] (13,-6) circle (.1cm) node[align=left,   above]{};
		\fill[black] (14,-5) circle (.1cm) node[align=left,   above]{};
		\fill[black] (15,-5) circle (.1cm) node[align=left,   above]{};
		\end{tikzpicture}
		\qquad
		\begin{tikzpicture}[scale=0.4]
		\draw[] (1,-9.5) -- (1,-11.5);
		\draw[] (1,-9.5) -- (-2,-9.5);
		\draw[] (-2,-11.5) -- (-2,-9.5);
		\draw[] (1,-9.5) -- (4,-9.5);
		\draw[] (4,-9.5) -- (4,-11.5);
		\draw[] (1,-11.5) -- (0,-11.5);
		\draw[] (1,-11.5) -- (2,-11.5);
		\draw[] (-2,-9.5) -- (-2,-9);
		\draw[] (4,-9.5) -- (4,-9);
		\draw (4,-8.5) circle (0.5cm);
		\draw (-2,-8.5) circle (0.5cm);
		\draw (-1,-11.5) ellipse (1cm and 0.2cm);
		\draw (3,-11.5) ellipse (1cm and 0.2cm);
		\node[draw] at (1,-12.5) {e};
		\fill[black] (1,-9.5) circle (.1cm) node[align=left,   above]{};
		\fill[black] (-2,-9.5) circle (.1cm) node[align=left,   above]{};
		\fill[black] (4,-9.5) circle (.1cm) node[align=left,   above]{};
		\fill[black] (1,-11.5) circle (.1cm) node[align=left,   above]{};
		\fill[black] (0,-11.5) circle (.1cm) node[align=left,   above]{};
		\fill[black] (2,-11.5) circle (.1cm) node[align=left,   above]{};
		\fill[black] (-2,-9) circle (.1cm) node[align=left,   above]{};
		\fill[black] (4,-9) circle (.1cm) node[align=left,   above]{};
		\fill[black] (-2,-11.5) circle (.1cm) node[align=left,   above]{};
		\fill[black] (4,-11.5) circle (.1cm) node[align=left,   above]{};
		\end{tikzpicture}
		\qquad
		\begin{tikzpicture}[scale=0.4]
		\draw[] (9,-9) -- (11,-9);
		\draw[] (9,-11) -- (11,-11);
		\draw[] (11,-9) -- (10,-10);
		\draw[] (10,-10) -- (11,-11);
		\draw[] (10,-10) -- (11,-10);
		\draw[] (11,-9) -- (14,-9);
		\draw[] (11,-11) -- (14,-11);
		\draw[] (14,-9) -- (14,-11);
		\draw[] (11,-10) -- (14,-9);
		\draw[] (11,-10) -- (14,-11);
		\draw[] (12,-9.66) -- (12,-10.34);
		\draw (8.5,-9) circle (0.5cm);
		\draw (8.5,-11) circle (0.5cm);
		\node[draw] at (10,-12) {f};
		\fill[black] (9,-9) circle (.1cm) node[align=left,   above]{};
		\fill[black] (11,-9) circle (.1cm) node[align=left,   above]{};
		\fill[black] (10,-10) circle (.1cm) node[align=left,   above]{};
		\fill[black] (9,-11) circle (.1cm) node[align=left,   above]{};
		\fill[black] (11,-11) circle (.1cm) node[align=left,   above]{};
		\fill[black] (11,-10) circle (.1cm) node[align=left,   above]{};
		\fill[black] (12,-9.66) circle (.1cm) node[align=left,   above]{};
		\fill[black] (12,-10.34) circle (.1cm) node[align=left,   above]{};
		\fill[black] (14,-9) circle (.1cm) node[align=left,   above]{};
		\fill[black] (14,-11) circle (.1cm) node[align=left,   above]{};
		\end{tikzpicture}
		\qquad
		\begin{tikzpicture}[scale=0.4]
		\draw[] (0,-15) -- (-1,-15);
		\draw[] (0,-15) -- (1,-14);
		\draw[] (0,-15) -- (1,-16);
		\draw[] (1,-14) -- (3,-15);
		\draw[] (1,-16) -- (3,-15);
		\draw[] (2,-15.5) -- (2,-14.5);
		\draw[] (1,-14) -- (5,-14);
		\draw[] (1,-16) -- (5,-16);
		\draw[] (3,-15) -- (4,-15);
		\draw[] (4,-15) -- (5,-14);
		\draw[] (4,-15) -- (5,-16);
		\draw[] (5,-14) -- (5,-16);
		\draw (-1.5,-15) circle (0.5cm);
		\node[draw] at (2,-17.5) {g};
		\fill[black] (1,-14) circle (.1cm) node[align=left,   above]{};
		\fill[black] (5,-14) circle (.1cm) node[align=left,   above]{};
		\fill[black] (1,-16) circle (.1cm) node[align=left,   above]{};
		\fill[black] (5,-16) circle (.1cm) node[align=left,   above]{};
		\fill[black] (0,-15) circle (.1cm) node[align=left,   above]{};
		\fill[black] (-1,-15) circle (.1cm) node[align=left,   above]{};
		\fill[black] (2,-14.5) circle (.1cm) node[align=left,   above]{};
		\fill[black] (2,-15.5) circle (.1cm) node[align=left,   above]{};
		\fill[black] (3,-15) circle (.1cm) node[align=left,   above]{};
		\fill[black] (4,-15) circle (.1cm) node[align=left,   above]{};
		\end{tikzpicture}
		\qquad
		\begin{tikzpicture}[scale=0.4]
		\draw[] (8,-16) -- (9,-16);
		\draw[] (9,-16) -- (10,-15.5);
		\draw[] (9,-16) -- (10,-16.5);
		\draw[] (10,-15.5) -- (10,-16.5);
		\draw[] (10,-16.5) -- (13,-16.5);
		\draw[] (10,-15.5) -- (11.5,-14);
		\draw[] (11.5,-13) -- (11.5,-14);
		\draw[] (13,-15.5) -- (11.5,-14);
		\draw[] (13,-15.5) -- (13,-16.5);
		\draw[] (13,-15.5) -- (14,-16);
		\draw[] (13,-16.5) -- (14,-16);
		\draw[] (14,-16) -- (15,-16);
		\draw (7.5,-16) circle (0.5cm);
		\draw (15.5,-16) circle (0.5cm);
		\draw (11.5,-12.5) circle (0.5cm);
		\node[draw] at (12,-17.5) {h};
		\fill[black] (8,-16) circle (.1cm) node[align=left,   above]{};
		\fill[black] (9,-16) circle (.1cm) node[align=left,   above]{};
		\fill[black] (10,-15.5) circle (.1cm) node[align=left,   above]{};
		\fill[black] (10,-16.5) circle (.1cm) node[align=left,   above]{};
		\fill[black] (11.5,-14) circle (.1cm) node[align=left,   above]{};
		\fill[black] (11.5,-13) circle (.1cm) node[align=left,   above]{};
		\fill[black] (13,-15.5) circle (.1cm) node[align=left,   above]{};
		\fill[black] (13,-16.5) circle (.1cm) node[align=left,   above]{};
		\fill[black] (14,-16) circle (.1cm) node[align=left,   above]{};
		\fill[black] (15,-16) circle (.1cm) node[align=left,   above]{};
		\end{tikzpicture}
		\caption{The eight trivalent planar graphs of genus 6, which are not tropically planar \cite{M19}, but which are not covered by Theorem~\ref{thm:obstructions}.}
		\label{fig:genus6_uncategorized}
	\end{figure}
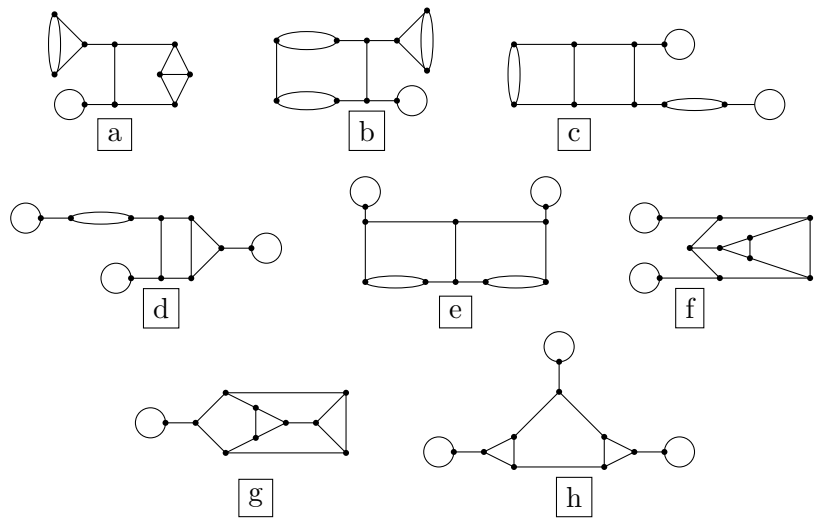

\bibliographystyle{siam}
\bibliography{biblio.bib}

\begin{thebibliography}{10}

\bibitem{chemicalgraphs}
{\sc A.~T. Balaban}, {\em Chemical applications of graph theory}, Academic
  Press, 1976.

\bibitem{JB15}
{\sc S.~Brodsky, M.~Joswig, R.~Morrison, and B.~Sturmfels}, {\em Moduli of
  tropical plane curves}, Research in the Mathematical Sciences, 2 (2015),
  p.~4.

\bibitem{CartwrightManjunathYao:2016}
{\sc D.~Cartwright, A.~Dudzik, M.~Manjunath, and Y.~Yao}, {\em Embeddings and
  immersions of tropical curves}, Collect. Math., 67 (2016), pp.~1--19.

\bibitem{CV09}
{\sc W.~Castryck and J.~Voight}, {\em On nondegeneracy of curves}, Algebra \&
  Number Theory, 3 (2009), pp.~255--281.

\bibitem{M19}
{\sc D.~Coles, N.~Dutta, S.~Jiang, R.~Morrison, and A.~Scharf}, {\em Tropically
  planar graphs}, 2019.
\newblock Preprint \arXiv{1908.04320v3}.

\bibitem{Triangulations}
{\sc J.~A. De~Loera, J.~Rambau, and F.~Santos}, {\em Triangulations}, vol.~25
  of Algorithms and Computation in Mathematics, Springer-Verlag, Berlin, 2010.
\newblock Structures for algorithms and applications.

\bibitem{DMV:polymake}
{\sc E.~Gawrilow and M.~Joswig}, {\em polymake: a framework for analyzing
  convex polytopes}, in Polytopes---combinatorics and computation (Oberwolfach,
  1997), vol.~29 of DMV Sem., Birkh\"auser, Basel, 2000, pp.~43--73.

\bibitem{HahnMarkwigRenTyomkin:2019.02440}
{\sc M.~A. Hahn, H.~Markwig, Y.~Ren, and I.~Tyomkin}, {\em Tropicalized
  quartics and canonical embeddings for tropical curves of genus 3},
  International Mathematics Research Notices,  (2019).
\newblock Published online, \doi{10.1093/imrn/rnz084}.

\bibitem{HJ08}
{\sc S.~Herrmann and M.~Joswig}, {\em Splitting polytopes}, M{\"u}nster J.
  Math, 1 (2008), pp.~109--141.

\bibitem{LamPostnikov05}
{\sc T.~Lam and A.~Postnikov}, {\em Alcoved polytopes. {I}}, Discrete Comput.
  Geom., 38 (2007), pp.~453--478.

\bibitem{Tropical+Book}
{\sc D.~Maclagan and B.~Sturmfels}, {\em Introduction to tropical geometry},
  vol.~161 of Graduate Studies in Mathematics, American Mathematical Society,
  2015.

\bibitem{M17}
{\sc R.~Morrison}, {\em Tropical hyperelliptic curves in the plane}, 2019.
\newblock Preprint \arXiv{1708.00571}.

\end{thebibliography}

\end{document}